\documentclass[12pt,reqno,oneside]{article}

\usepackage[letterpaper,hmargin=1in,vmargin=1in]{geometry}
\usepackage{amsmath}
\usepackage{amsthm}
\usepackage{amssymb}
\usepackage{mathtools}
\usepackage{fourier}
\usepackage{mathrsfs}
\usepackage{tikz}
\usetikzlibrary{
  matrix,
  arrows,
  arrows.meta,
  angles,
  shapes,
  backgrounds,
  fit,
  cd,
  calc,
  positioning,
  intersections,
  decorations.markings,
  decorations.pathmorphing,
  backgrounds,patterns,
  decorations.pathreplacing}
\usetikzlibrary{external}
\usepackage{pgfplots}
\usepackage{tikz-3dplot}
\usetikzlibrary{
  decorations.fractals,
  spy,
  pgfplots.colorbrewer,
  pgfplots.fillbetween
}

\pgfplotsset{
  compat=1.18,
  title style={align = center},
  y tick label style={/pgf/number format/.cd,fixed,precision=5,/tikz/.cd},
  scaled y ticks = false,
  cycle list/Paired,
  axis lines=left
}

\usepackage[hidelinks,breaklinks]{hyperref}
\usepackage{breakurl}
\usepackage{xcolor}
\definecolor{darkgreen}{rgb}{0,0.45,0}
\hypersetup{colorlinks,urlcolor=blue,citecolor=darkgreen,linkcolor=darkgreen,linktocpage}
\usepackage[capitalize]{cleveref}
\usepackage[style=alphabetic,maxnames=10,backend=bibtex,hyperref=true]{biblatex}
\usepackage{authblk}
\usepackage{comment}
\usepackage{ifthen}
\usepackage{multicol}
\usepackage[justification=centering]{caption}

\mathtoolsset{mathic=true}

\makeatletter
\renewcommand\subsubsection{\@startsection{subsubsection}{3}{\z@}%
                                     {-3.25ex\@plus -1ex \@minus -.2ex}%
                                     {-1.5ex \@plus -.2ex}% Formerly 1.5ex \@plus .2ex
                                     {\normalfont\normalsize\bfseries}}
\makeatother
\numberwithin{equation}{section}
\newtheorem{definition}{Definition}[section]

\newtheorem{lemma}[definition]{Lemma}
\newtheorem{proposition}[definition]{Proposition}
\newtheorem{corollary}[definition]{Corollary}
\newtheorem{theorem}[definition]{Theorem}
\newtheorem{conjecture}[definition]{Conjecture}

\theoremstyle{remark}
\newtheorem{remark}[definition]{Remark}
\newtheorem{example}[definition]{Example}

\newcommand{\define}[1]{\textbf{#1}}

\newcommand{\norm}[1]{\left\lVert#1\right\rVert}
\DeclareMathOperator{\Mid}{mid}

\newcommand{\R}{\mathbb{R}}
\newcommand{\C}{\mathcal{C}}
\newcommand{\D}{\mathcal{D}}

\newcommand{\vect}{\mathbf{Vec}}

\DeclareMathOperator{\End}{End}

\newcommand{\id}{\mathsf{id}}

\DeclareMathOperator{\cost}{cost}
\newcommand{\I}{\mathsf{I}}

\newcommand{\di}{d_{\I}}

\newcommand{\Int}{\mathbf{Int}}

\newcommand{\cL}{\mathcal{L}}
\newcommand{\subR}{\mathcal{L}(\mathbb{R})}
\newcommand{\sub}{\mathbf{Sub}}
\newcommand{\mon}{\mathbf{Mon}}
\newcommand{\powR}{\mathcal{P}(\mathbb{R})}
\newcommand{\cG}{\mathcal{G}}
\newcommand{\cH}{\mathcal{H}}

\newcommand{\snap}{\mathcal{S}}

\DeclareMathOperator{\diam}{diam}

\DeclareMathOperator{\n}{\mathbf{n}}
\DeclareMathOperator{\e}{\mathbf{e}}

\DeclareMathOperator{\Mat}{\mathsf{Mat}}
\DeclareMathOperator{\Bool}{\mathsf{Bool}}
\DeclareMathOperator{\Fun}{\mathsf{Fun}}
\DeclareMathOperator{\Path}{\mathsf{Path}}
\usepackage{comment}

\def\noteson{
    \gdef\note##1{\noindent{\color{blue}[##1]}}
    \gdef\todo##1{\noindent{$\dagger$ {\color{red}##1} $\dagger$}}
    \gdef\robert##1{\noindent{Robert : \color{green!40!black}[##1]}}
    \gdef\jacob##1{\noindent{Jacob : \color{red!80!black}[##1]}}
    \gdef\justin##1{\noindent{Justin : \color{green!40!black}[##1]}}
    \gdef\robby##1{\noindent{Robby : \color{green!40!black}[##1]}}
    \gdef\alan##1{\noindent{Alan : \color{green!40!black}[##1]}}
    \gdef\tung##1{\noindent{Tung : \color{green!40!black}[##1]}}
    \gdef\brendan##1{\noindent{Brendan : \color{green!40!black}[##1]}}
    \gdef\bob##1{\noindent{Bob : \color{green!40!black}[##1]}}
    \gdef\prov##1{{\color{gray} ##1}} % provisionary
    \gdef\precise##1{\noindent{\color{gray} ##1}}
}

\noteson
\everymath{\displaystyle}

\title{Algebraic and Geometric Models for Space Networking}

\author[1]{William Bernardoni}
\author[2]{Robert Cardona}
\author[3, 5]{Jacob Cleveland}
\author[2]{Justin Curry}
\author[2]{Robert Green}
\author[2]{Brian Heller}
\author[4]{Alan Hylton}
\author[2]{Tung Lam}
\author[5]{Robert Kassouf-Short}

\affil[1]{Case Western Reserve University}
\affil[2]{University at Albany, State University of New York (SUNY)}
\affil[3]{Colorado State University}
\affil[4]{NASA Goddard Space Flight Center}
\affil[5]{NASA Glenn Research Center}

\begin{document}

\maketitle

\begin{abstract}
In this paper we introduce some new algebraic and geometric perspectives on networked space communications. Our main contribution is a novel definition of a time-varying graph (TVG), defined in terms of a matrix with values in subsets of the real line $\powR$. We leverage semi-ring properties of $\powR$ to model multi-hop communication in a TVG using matrix multiplication and a truncated Kleene star. This leads to novel statistics on the communication capacity of TVGs called lifetime curves, which we generate for large samples of randomly chosen STARLINK satellites, whose connectivity is modeled over day-long simulations.
Determining when a large subsample of STARLINK is temporally strongly connected is further analyzed using novel metrics introduced here that are inspired by topological data analysis (TDA).
To better model networking scenarios between the Earth and Mars, we introduce various semi-rings capable of modeling propagation delay as well as protocols common to Delay Tolerant Networking (DTN), such as store-and-forward.
Finally, we illustrate the applicability of zigzag persistence for featurizing different space networks and demonstrate the efficacy of K-Nearest Neighbors (KNN) classification for distinguishing Earth-Mars and Earth-Moon satellite systems using time-varying topology alone.
\end{abstract}
\newpage
\tableofcontents

% % REQUIRED
% \begin{keywords}
%     semi-rings, 
%     kleene star,
%     persistent homology,
%     topological data analysis,
%     zigzag persistence,
%     time-varying graphs, 
%     space communications,
%     delay-tolerant networking,
%     path-routing algebra,
%     machine learning
%     %\LaTeX
% \end{keywords}

% % REQUIRED
% \begin{MSCcodes}
% 55N31 (Primary), 68M10, 16Y60 
% \end{MSCcodes}

% 2 sec:sem_rng_mod
%% 2.1 sec:semi-ring-matrix
%% 2.2 sec:lifetime-curves
%% 2.3 sec:propagation-delay
%% 2.4 sec:semi-ring-review-graph-optimization
%%% 2.4.1 sec:semi-ring-review-graph-optimization
%%% 2.4.2 sec:tropical-endo
%% 2.5 sec:UCS
% 3 sec:TDA-and-metrics
%% 3.1 sec:distances-on-TVGs
%%% 3.1.1 sec:distances-fixed-node
%%% 3.1.2 sec:distances-w-unknown-node
%%% 3.1.3 sec:connect-w-interleaving-cosheaves
%% 3.2 sec:top-summaries-TVGs
%%% 3.2.1 sec:KNN
% 4 sec:future-dir

\subsubsection*{Funding Acknowledgement:}
This work is a product of NASA Contract 80GRC020C0016.

\section{Introduction and Outline}

%In this paper we introduce a new set of mathematical tools that we believe are useful for modeling and generalizing current approaches to space networking.
As humanity embarks on its next steps in space exploration, with international and commercial actors providing a huge influx of new space assets, the need to automate and scale space communications has become increasingly pressing.
Current communication in space---between, say, a rover on Mars and a particular building on Earth---is handled by teams of engineers manually scheduling which point-to-point links are to be used at what times to provide end-to-end transmission of data. 
Once a message reaches a ground station on Earth traditional terrestrial networking theory, such as TCP/IP, takes over, which relies heavily on low latency and a largely static network architecture.
However, neither of these requirements hold in space. 
Special relativity dictates that all communication is constrained by the speed of light, which means the fastest possible one-way communication between Earth and Mars is constrained to 3 and 22 light-minutes, depending on their relative positions in orbit.
Networking topology can also evolve rapidly, as illustrated by satellites in low-Earth orbit (LEO), whose orbits are typically 90 minutes and a random pair of satellites may have a line-of-sight connection for only a few minutes.
In all these scenarios occlusions typically force a break in point-to-point communication, which is handled by either local storage or the use of an alternate route that circumvents the occlusion.

Motivated by these problems, we introduce here a novel set of tools for modeling time-varying networks that appear in actual space networking situations.
Our main contribution is a clear definition of a time-varying graph (TVG), which we view as a matrix whose entries are populated by subsets of time.
We exploit the fact that subsets of time, viewed as elements of $\powR$, have a natural notion of addition and multiplication given by union and intersection, thereby making $\powR$ into a semi-ring; see \Cref{defn:semi-ring}.
With this observation in hand, we model end-to-end communication capacity more accurately by considering the Kleene star (\Cref{def:kleene_star}) of our TVG adjacency matrix:
\[
A^*:= I + A + A^2 + \cdots + A^k + \cdots
\]
Since addition in the above equation represents entry-wise union of times of connectivity, the partial series sum $C_k(A):=I+\cdots + A^k$ provides an interesting filtration parameter that accumulates windows of opportunity along walks of length $k$ or less.
By considering the average measure of each entry $C_k(A)_{i,j}$, measured over some fixed window of time $W\subset \R$, we obtain a novel statistic that we call the \emph{lifetime curve} in \Cref{lem:lifetime-curves}, which we use as a proxy to measure how close a TVG is to a strongly connected one\footnote{An ordinary directed graph is strongly connected if you can go from any node to any other node. A TVG is then strongly connected if you can go from any node to any other node at any time.}.

We implement these ideas in code (available at \url{github.com/TheaMAS/sat-parser}) and simulate various networking scenarios with the help of Satellite Orbital Analysis Program (SOAP), which is a tool for accurately simulating orbital mechanics and calculating windows of opportunity for line-of-sight communication\footnote{For our analysis, all possible line-of-sight opportunities are considered simultaneously valid. In practice, a given node might be able to only establish one link at time, meaning choosing one contact necessarily precludes others. This, among other considerations, should lead to interesting and well-defined future research topics.}.
By using the large database of STARLINK satellites available on \url{celestrak.org}, we simulate random LEO networks by taking different size samples from the STARLINK network.
As illustrated in \Cref{fig:lifetime-curves-sample} these lifetime curves can have radically different shapes, but they seem to undergo a clear phase transition when the number of nodes exceeds $n=40$, suggesting that more than 40 nodes are needed to ensure strong connectivity in a LEO network.
The shapes of these curves are currently not well understood and motivate further mathematical research into their structure, cf. \Cref{conjecture:r=3}. 

\subsection{Outline for the Applications-First Reader}

For the reader who is primarily interested in the immediate application of our methods to space networking, and in particular how to measure the difference between near-Earth and deep-space communication networks, we advise that they proceed directly from the beginning of \Cref{sec:sem_rng_mod} (including \Cref{sec:semi-ring-matrix} and \Cref{sec:lifetime-curves}) to the start of \Cref{sec:TDA-and-metrics}.
There we continue the question of how close a given TVG is to a strongly connected one, by introducing a bona fide distance on TVGs that is inspired by topological data analysis (TDA).
Here we leverage the perspective that times when an edge exists (or does not exist) can be represented as a collection of intervals, assuming the TVG is not too pathological in its connectivity. 
Such intervals can be represented using a persistence diagram and there are many well-defined distances on persistence diagrams.
To emphasize that the distances in \Cref{sec:distances-fixed-node} and \Cref{sec:distances-w-unknown-node} are not being used on the output of a traditional persistent homology pipeline---because no homology is being taken---we call these distances the \emph{disconnect distances} in \Cref{def:p-q-disconnect-distance}.
We then use these distances to measure more carefully the connectivity properties of sub-samples of STARLINK, shown in \Cref{fig:STARLINK-sample-distances}, which shows that 100 nodes is more likely needed to establish strong connectivity.
Moving beyong STARLINK, we illustrate how to compare Earth-Moon and Earth-Mars systems using zigzag persistence \cite{Carlsson:zigzag:2009,carlsson2010zigzag}, which is another tool borrowed from TDA that summarizes how network topology varies over time, and is agnostic to node labels. We show that, when our TVGs are featurized using $H_1$ zigzag persistence barcodes, how a K-Nearest Neighbor (KNN) classifier can be used to automatically distinguish types of space networks. As outlined in \Cref{sec:future-dirs} we take this as the first step in creating an automatic recommendation and segmentation protocol for space internet that uses machine learning.

\subsection{Outline for the Applied Topology Reader}

For the reader who is interested in how our work expands the theory of TDA, we recommend proceeding directly from \Cref{sec:lifetime-curves} and reading the entirety of \Cref{sec:TDA-and-metrics}.
Our definition of TVGs leads naturally to a summary cosheaf (\Cref{def:summary-cosheaf}) that refines the Reeb cosheaf of \cite{crg} by tracking the entire graph structure and not just the connected components.
However, similar to \cite{crg} we adapt the interleaving distance construction to this setting and prove some novel isometry results (\Cref{thm:isometries}) that are very close in spirit to the results on merge trees proved in \cite{gasparovic2019intrinsic}.
Our summary cosheaf also provides a novel pipeline for proving stability of the zigzag barcodes under the aforementioned metrics, see \Cref{prop:TVG-barcode-stability}. 

%\subsubsection{Prior Work on TDA for Time-Varying Graphs}
%\todo{TUNG, please cite papers whose methods are very close to ours, e.g.
%\cite{temporalhypergraph:2023} and Woojin's paper \cite{kim2017extracting} and Munch \cite{myers2023temporal}}
\subsubsection*{Prior Work in TDA on Time-Varying Graphs}

Zigzag persistence \cite{Carlsson:zigzag:2009,carlsson2010zigzag}, which is a common tool from TDA, has been used to study time-evolving topology in dynamic networks for several years. 
At a high-level, zigzag persistence reveals the formation, disappearance and duration of topological features of a space $X$ that is parameterized by $\R$, i.e., by using a function $f:X\to\R$. 
To understand how topology at one time $t_0$ is related to the topology at another time $t_1$, one typically considers the alternating inclusion
$f^{-1}(t_0) \hookrightarrow f^{-1}[t_0,t_1] \hookleftarrow f^{-1}(t_1)$
and takes homology, which is what gives zigzag persistence its name.
In order to analyze temporal topological features via zigzag persistence, a collection of snapshots of a temporal network are considered together with their zigzag unions (or intersections) to form a sequence of simplicial complexes. 
Network snapshots can be obtained via a sliding window or temporal partitioning construction\cite{Lozeve:2018}. 
This approach is used in \cite{myers2023temporal} to analyze the Great Britain transportation network, for example. 
%and the ordinal partition temporal network where the network snapshots are graphs. 
A similar approach, which is used in \cite{temporalhypergraph:2023} to analyze the social networks and cyber data of various communities, considers situations where the network snapshots are hypergraphs. 
In the context of space communication, the study \cite{Hylton2022Survey} uses zigzag persistence on simulated space networks to identify and extract subnetwork structures in order to reduce the complexity of the network.
A more theoretical approach to dynamic graphs, carried out by Kim and M{\' e}moli \cite{kim2017extracting}, uses the M\"obius inversion perspective on persistence to define a novel summary of time-evolving clusters called the \emph{persistence clustergram}.

\subsection{Outline for the Applied Algebra Reader}

Finally, for the reader most interested in the semi-ring aspects of our work, we encourage them to read \Cref{sec:sem_rng_mod} in its entirety.
Even the more applied reader will benefit from seeing how our semi-ring model for TVGs allows us to model propagation delay in \Cref{sec:propagation-delay}.
However, the most theoretically-minded should turn their attention to \Cref{sec:UCS}, as this provides a sort of ``universal'' semi-ring that allows us to model all known aspects of graph optimization problems, including tropical geometry.
From the applied perspective, this section is most significant for its ability to model store-and-forward routing behavior, which is one of the most popular protocols in delay tolerant networking.

\subsubsection*{Prior Work on Algebraic Path Problems}

To our knowledge, the semi-ring models of time-varying graphs advanced here have not been considered before.
In particular, we believe that the matrix TVG semi-ring, the propagation delay semi-ring, and the universal contact semi-ring are all new, with no clear presence in the literature.
We say this with much trepidation as the semi-ring perspective on networks has at least a 50 year history, with Carr\'e's pioneering work \cite{Carre1971} as one of the first major papers. 
This perspective has endured, with textbook length treatments in \cite{Carre1981,heidergott2006max,gondran2008graphs} and \cite{Baras2010}. 

In all these works, one wants to view any network characterization or optimization problem as a matrix with entries in a carefully chosen semi-ring.
Solutions to these characterization/optimization problems are usually calculated via the transitive closure of that matrix, i.e., the Kleene star.
For instance, and this is expanded on in more detail in \Cref{sec:semi-ring-review-graph-optimization}, the shortest path problem on a network can be formulated use the tropical semi-ring $([0, \infty], \min, +)$. 
One compares the costs along all paths using the $\min$ (addition) operation, while each path cost is the result of aggregating arc weights along the path using the $+$ (multiplication) operation. 
The Kleene star for this matrix then consists of the lengths of pairwise shortest paths between nodes in networks, thus solving the all-pairs shortest-path problem with remarkable algebraic efficiency \cite{Mehryar2002}.
As Carr\'e observed in \cite{Carre1971}, these matrix-theoretic methods provably generalize direct methods of solving routing problems such as Bellman-Ford's and Floyd–Warshal's algorithms.
These observations were continued by \cite{Minoux1976}, who showed that generalized path algebras can be used to obtain the well-known pathfinding algorithms by Moore, Ford, and Dijkstra; see \cite{minoux:1998,minoux:1998b} for more on this.

%\jacob{What are closed and open networks?}
Surprisingly, our application of semi-rings to the internet-at-large is not new, as the works \cite{griffin2005metarouting,sobrinho2003network, sobrinho2001algebra, sobrinho2005dynamic, griffin2008bisemigroup, dynerowicz2013forwarding} study internet routing protocols from this perspective.
However each of these works are concerned with so-called ``closed networks,'' where the nodes and edges are fixed from the outset.
More recent work, coming out of the Applied Category Theory (ACT) community, has developed a more flexible theory for networks that allows for node discovery and for routing between other, unknown networks; see the works of Jade Master \cite{master:2020, master:2022, master:thesis}, whose thesis \cite{master:thesis} made major headway in the theory of these so-called ``open networks.''  We anticipate that this work will be important in future work on developing generalized, composable algorithms for routing in space.

\section{Algebraic Models for Time-Varying Graphs (TVGs)}\label{sec:sem_rng_mod}
%\tung{Barras'text on TVG will be mentioned here}

Time-varying graphs (TVGs) have been studied extensively by many different groups of people, but there is no single agreed-upon definition of what a TVG should be.
One model of a TVG is that a time-varying graph is simply a graph sequence $G_0, \ldots, G_n$, but this perspective obscures relationships across time.
Another model for a TVG (or temporal graph) \cite{bumpus2022edge}, which addresses this criticism, argues that a TVG is a graph $G=(V,E)$ equipped with a function $\tau:E \to 2^{\mathbb{N}}$ that specifies which indices in a graph sequence an edge lives.

\begin{figure}[h!]
    \centering
    \includegraphics[width=\textwidth]{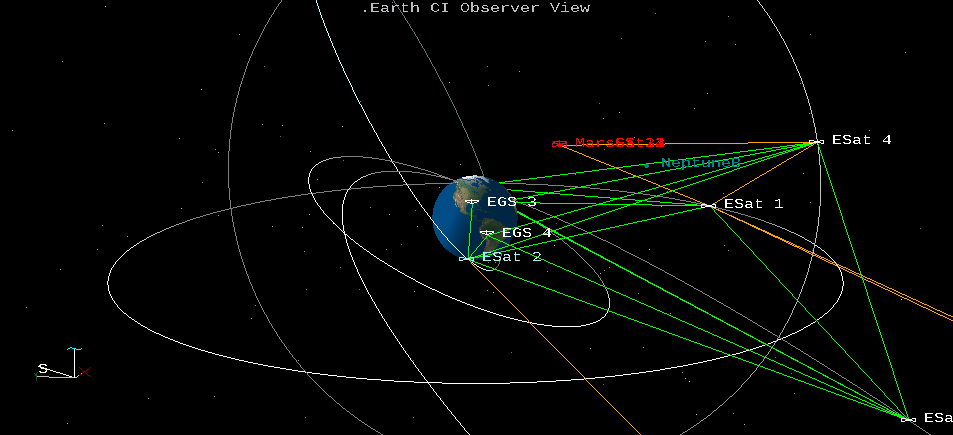}
    \caption{A screenshot from the Sattelite Orbital Analysis Program (SOAP) illustrates lines of sight between ground stations and satellites around Earth and Mars.}
    \label{fig:SOAP-Earth-Crop-1}
\end{figure}

Both of these perspectives are inappropriate for the purposes of space networking, where we are primarily interested in when two assets (ground stations, rovers, satellites, etc.) have a clear line of sight for communication; see \Cref{fig:SOAP-Earth-Crop-1}.
Due to orbital mechanics, each line of sight in a space network starts at a \emph{sunrise} time and ends with a \emph{sunset} time, which together marks the boundary of a single interval of connectivity.
Since the metric properties of these intervals---\emph{when} and \emph{for how long?}---are crucial for determining when and how much data can be routed across our network, we introduce a framework for talking about lifetimes of connections. It will be useful to organize the collection of connection lifetimes into a poset structure organized by inclusion:

\begin{definition}[Poset of Lifetimes]\label{def:lifetime-poset}
Let $\subR$ be the \define{poset of lifetimes}. 
A non-trivial element $a\in \subR$ is 
a finite union of disjoint closed intervals
\[
a = [x_0,y_0] \cup \cdots \cup [x_n,y_n],
\]
where $x_0\in \R \cup \{-\infty\}$ and $y_n\in \R\cup \{+\infty\}$, with $x_n\leq y_n \leq x_{n+1}$ for all $n$. 
The partial order on lifetimes is given by inclusion of subsets, i.e., $a\preceq b$ if $a\subseteq b$.
\end{definition}

\begin{figure}
    \centering
    \includegraphics[width=\textwidth]{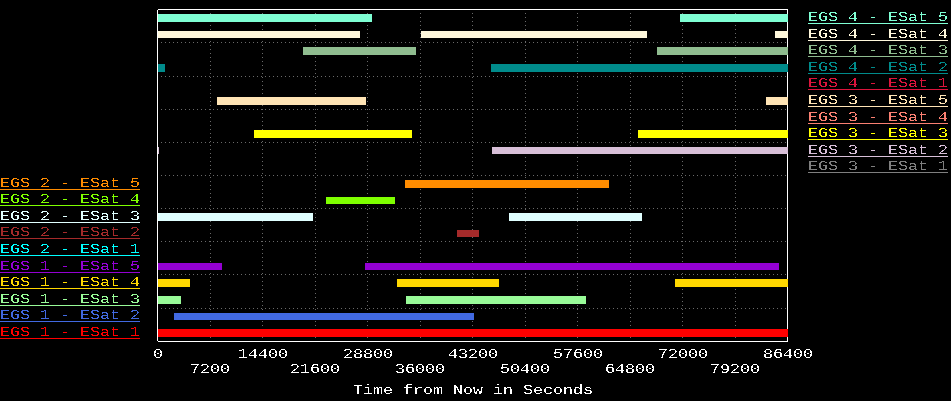}
    \caption{Sunrise and sunset times for line of sight communication in \Cref{fig:SOAP-Earth-Crop-1} define intervals of connectivity, which is what our definition of a TVG emphasizes.}
    \label{fig:SOAP-Earth-Crop-2}
\end{figure}

% \begin{remark} We note that the semi-ring $\subR$ is idempotent: for $a \in \subR$, $a + a = a \cup a = a$. A partial order $\succcurlyeq$ on the idempotent semi-ring $\subR$ can be defined as follows: \[ \forall a, b \in \subR: a \succ b  \iff a + b = a, \text{ and } a \neq b\]
% \end{remark}

Our definition of a time-varying graph is a graph that is equipped with an order-reversing map to the poset of lifetimes $\subR$. 
More precisely, we have the following definition.

\begin{definition}[TVG]\label{defn:TVG}
Let $G$ be a graph, i.e., a set of edges $E$ and a set of vertices $V$ with an incidence relation $v<e$ that indicates when a vertex $v$ belongs to an edge $e$.
A \define{time-varying graph (TVG)} $\cG=(G,\ell_M)$ is a graph $G$ along with an order-reversing \define{lifetime function} 
\[
\ell_M:(G,<)^{op} \to \subR, \quad \text{i.e.} \quad \text{if} \quad v <e \quad \Rightarrow \quad  \ell_M(e) \subseteq \ell_M(v).
\]
The assumption that $\ell_M(e) \subseteq \ell_M(v)$ is the \define{containment axiom}, as it requires that a vertex be alive whenever an edge is alive.
\end{definition}

%\todo{insert a simple figure illustrating the reverse containment}

\begin{remark}
We remark that any adjective for graphs descends to a modifier on time-varying graphs. 
For example, \Cref{defn:TVG} generalizes to hypergraphs as it only requires the notion of incidence (or containment) of vertices into hyperedges; see \cite{temporalhypergraph:2023}, which uses temporal attribution for hypergraph applications. 
Additionally, if $G$ is a directed multigraph, then we can define $v<e$ iff $v$ is the head or tail of the edge $e$, but directedness of an edge $e$ that goes from $i$ to $j$. 
Finally, recall that a (directed) graph is \define{simple} if there exists at most one (directed) edge between any two vertices.
\end{remark}

% We will restrict our attention to simple (directed) graphs $G$ for this paper and make precise hypotheses as needed.
% For now, we will use this assumption so that we can work with matrices associated to TVGs.
% Alternatively, one can take the following definition as our logical starting point and generalize only later, when necessary.

For simplicity, we will work with an alternative matrix formulation for TVGs, which can be viewed as the primary definition for this paper. Note that in this approach, the poset is specialized to a total order.

\begin{definition}[Lifetime Matrices and General Matrix TVGs]\label{defn:matrix-TVG}
Every simple directed TVG ${\cG=(G,\ell_M)}$ that is equipped with a total order on its vertex set $V$ has a representative \define{matrix of lifetimes}:
\[
M : V\times V \to \subR \qquad (i,j) \mapsto M(i,j)=\ell_M(i,j) \subseteq \R.
\]
The collection of all lifetime matrices, written $\Mat_{n}(\subR)$, includes every possible matrix with entries in $\subR$, without any containment axiom.
We call a matrix whose entries are arbitrary subsets of $\R$, i.e., $M\in \Mat_{n}(\powR)$, a \define{matrix TVG}.
\end{definition}

\Cref{defn:matrix-TVG} should remind the reader of the adjacency matrix in graph theory, which for a simple graph has $0$s along the diagonal and $1$s whenever an edge from $i$ to $j$ exists. 
One important difference is that we typically assume that a lifetime matrix $M$ has $\R$s along the diagonal, which, by analogy would say that every vertex has a self-loop sitting over it. This is a reasonable interpretation in the setting of message passing across a time-evolving graph, but has some drawbacks when trying to emulate (or generalize) more traditional constructions in graph theory.
To allow ourselves to work with both conventions, we introduce the following definition.

\begin{definition}[Adjacency Matrix of a TVG]\label{defn:adjacency-matrix-TVG}
The \define{adjacency matrix} of a simple directed TVG $\cG=(G,\ell_M)$, written $A$, is identical to the associated matrix TVG $M$ except that $A_{ii}=\varnothing$.
\end{definition}

\subsection{Semi-Ring and Matrix Perspectives on Graphs and TVGs}\label{sec:semi-ring-matrix}

One of the important reasons for working with matrices is that certain operations---such as addition and multiplication---are available on matrices, which are not obvious for graphs.
In this section we recall the interpretation of simple directed graphs as a matrix of Booleans and how the higher powers of the matrix of Booleans models walks of corresponding length.
This then clears the way for defining and interpreting these matrix operations for TVGs.
%Additionally, these operations introduce a semi-ring structure on TVGs, which generalizes the classical Boolean perspective on directed graphs.
%The two operations of matrix addition and multiplication are inherited from the semi-ring structure on the poset of lifetimes, as the next proposition explains.
At the heart of both of these perspectives is the observation that addition and multiplication are very general operations and are formally unified in the language of semi-rings, which are rings with negatives removed.

\begin{definition}\label{defn:semi-ring}
A \define{semi-ring} $(S,\oplus,\otimes,\n,\e)$ consists of a set $S$ with an addition operation $\oplus$ and a multipliciation operation $\otimes$ along with neutral elements $\n$ (the ``0'' element) and $\e$ (the ``1'' element) for these operations. These operations need to further satisfy the following four collections of identities:
\begin{enumerate}
\item $(S,\oplus,\n)$ is a commutative monoid (see \Cref{def:monoid}) with identity element $\n$. This means:
\begin{itemize}
\item $(a\oplus b)\oplus c=a\oplus (b\oplus c)$
\item $a \oplus b = b \oplus a$
\item $\n\oplus a=a=a\oplus \n$
\end{itemize}
\item $(S,\otimes)$ is a (not necessarily commutative) monoid with identity element $\e$. This means
\begin{itemize}
\item $(a\otimes b)\otimes c=a\otimes (b\otimes c)$
\item $\e\otimes a=a=a\otimes \e$
\end{itemize}
\item Multiplication distributes over addition on both the left and right:
\begin{itemize}
\item $a\otimes (b\oplus c)=(a\otimes b)\oplus(a\otimes c)$
\item $(a\oplus b)\otimes c=(a\otimes c)\oplus(b\otimes c)$
\end{itemize}
\item Multiplication by $\n$ (the ``$0$'' element) annihilates $S$:
\begin{itemize}
\item $\n \otimes a=\n = a \otimes \n$
\end{itemize}
\end{enumerate}
\end{definition}

\begin{remark}
    Since semi-rings are rings with negatives removed, some authors prefer to call semi-rings \textbf{rigs}---because they are ``rings'' with the ``n'' (as in ``negatives'') removed.
\end{remark}

The prototypical example of a semi-ring is the set of natural numbers $\mathbb{N}=\{0,1,2,\ldots\}$ with normal addition and multiplication.
We will work primarily with the following semi-rings, which are pertinent to the study of TVGs.
%\todo{Add reference to Baras and re-explain, if necessary, connections to networking theory.}
%However, there are more examples of semi-rings that are pertinent to graph theory, which are reviewed below.
%However, the various matrix operations on graphs in these settings rests on the fact that matrices form a semi-ring as well.

\begin{definition}[Boolean Semi-Ring]\label{defn:boolean-semi-ring}
 Consider the set $\Bool=\{\bot,\top\}$ with $\n=\bot=0$ corresponding to \texttt{FALSE} and $\e=\top=1$ corresponding to \texttt{TRUE}. If we define $\oplus=\vee$ to be the logical \texttt{OR} operation and $\otimes=\wedge$ to be the logical \texttt{AND} operation, then $(\Bool,\vee,\wedge,\bot,\top)$ defines the \define{Boolean semi-ring}.
\end{definition}

% In graph optimization, typically one wants to find shortest paths. This amounts to working with an adjacency matrix whose entries are in the tropical (min-plus) semi-ring.

% \begin{example}[Tropical semi-ring]
% Consider the set $\mathbb{T}=\mathbb{R}\cup\{\infty\}$. If we define $a\oplus b=\min\{a,b\}$ and $a\odot b=a+b$, then the set $\mathbb{R}\cup\{\infty\}$ admits the structure of a semi-ring, denoted $\mathbb{T}$, called the \define{tropical} or \define{min-plus semi-ring}.
% \end{example}
%\todo{consider putting walk definition above Definition 2.4}

\begin{definition}[Path Semi-Ring]\label{defn:path-semi-ring}
    Suppose $G=(V,E)$ is a simple directed graph. 
    Let $\Path(G)$ be the set of formal combinations of paths (or \define{walks}) in $G$ of arbitrary length, i.e., 
    \[\Path(G) =\{ \sum_i \gamma_i \mid \gamma_i=[v_{i_0}, \ldots, v_{i_n}], \quad \text{where} \quad \forall i,j \, (v_{i_j}, v_{i_{j+1}}) \in E \}.
    \]
    If $a=\sum_i \gamma_i$ and $b=\sum_j \gamma_j$, then $a+b$ is the Boolean sum (i.e., set-theoretic union) of the paths in $a$ with the paths in $b$.
    The additive neutral element $\n$ is the empty collection of paths $\varnothing$.
    For multiplication, we first define concatenation of paths $\gamma_i\ast \gamma_j$ to be the extension of the path $\gamma_i$ by $\gamma_j$ if the last vertex in $\gamma_i$ is the first vertex in $\gamma_j$, otherwise it is $\varnothing$; notice this multiplication is not commutative.
    Extending linearly allows us to define 
    \[
    a\ast b =(\sum_i \gamma_i)\ast (\sum_j \gamma_j) = \sum_{ij}\gamma_i \ast \gamma_j, \quad \text{where} \quad \e=1= \sum_{v_i\in V} [v_i]
    \]
    is a multiplicative neutral element.
\end{definition}

%The perspective introduced in this paper is that TVGs can be viewed as matrices of lifetimes, which also form a semi-ring.

\begin{definition}[Lifetime and Powerset Semi-Ring]\label{def:lifetime-semi-ring}
    The poset of lifetimes $\subR$ (\Cref{def:lifetime-poset}) is a commutative semi-ring. This is actually a sub-semi-ring of the bigger semi-ring consisting of the powerset of $\R$, written $\powR$. 
    Addition and multiplication are defined as 
\begin{itemize}
    \item $a + b := a \cup b$, with neutral element $\n=\varnothing$ being the empty set, and
    \item $a\cdot b:=a \cap b$, with neutral element $\e=\R$ being the whole set $\R$.
\end{itemize}
More generally, the powerset $\mathcal{P}(X)$ of any set $X$ forms a commutative semi-ring with the same operations and analogous neutral elements.
\end{definition}

\begin{remark}[Idempotency and Partial Orders]\label{rmk:idempotent}
    Semi-rings such as $\Bool$ and $\mathcal{P}(X)$ have the additional property of being \define{idempotent}, i.e., for all elements $a$ we have $a+a=a$.
    Idempotent semi-rings give rise to a natural partial order, where
    \[
        a \preceq b \iff \exists c \text{ s.t. } a+c = b.
    \]
\end{remark}

\begin{definition}[Function Semi-Ring]
    Given any set $X$ and any semi-ring $(S,\oplus,\otimes,\n,\e)$, the set of functions from $X$ to $S$, written $\Fun(X,S)$, inherits the structure of a semi-ring. Given $m,n:X \to S$,
    \begin{itemize}
        \item we have $m\oplus n: X \to S$, where $(m\oplus n)(x):=m(x)\oplus n(x)$, and
        \item $m \otimes n : X \to S$, where $(m\otimes n)(x):=m(x) \otimes n(x)$.
    \end{itemize}
    The ``0'' function $0:X\to S$ with constant value $\n$ and the ``1'' function $1:X\to S$ with constant value $\e$ are the corresponding neutral elements.
\end{definition}

\begin{remark}[Boolean Functions and Subsets]\label{rmk:boolean-subsets}
    There is a close relationship between the semi-ring of subsets $\mathcal{P}(X)$ and the semi-ring of Boolean functions, denoted by $\Fun(X,\Bool)$.
    In fact, there is a map $\Phi$ that takes each subset $A\subseteq X$ to the indicator function $1_A:X \to \Bool$; this being the function that is $\top$ on points in $A$ and $\bot$ on points in $X - A$.
    This map also preserves addition and multiplication, i.e., 
    \begin{itemize}
        \item $\Phi(A\cup B)=1_{A \cup B}=1_A\oplus 1_B=\Phi(A)\oplus\Phi(B)$, and
        \item $\Phi(A\cap B)=1_{A\cap B} = 1_A\otimes 1_B=\Phi(A)\otimes \Phi(B)$.
    \end{itemize}
    This map $\Phi$ also preserves the neutral elements of both semi-rings, i.e., $\Phi(\varnothing)=0$ and $\Phi(X)=1$.
    All of this makes $\Phi$ a surjective \define{semi-ring homomorphism} (\Cref{def:semi-ring-homomorphism}) with trivial kernel, which is a \define{semi-ring isomorphism}.
\end{remark}

Finally, matrices valued in a semi-ring define another semi-ring.

\begin{definition}[Matrix Semi-Ring]\label{def:matrix-semi-ring}
    If $(S,\oplus,\otimes,\n,\e)$ is a semi-ring, then the collection of $n\times n$ matrices with entries in $S$, written $\Mat_{n}(S)$, is a semi-ring as well where $M+N$ is defined entry-wise as $(M+N)_{ij}:=M_{ij}\oplus N_{ij}$ and $(MN)_{ij}=\oplus_{k} M_{ik}\otimes N_{kj}$.
\end{definition}

\begin{figure}[h!]
    \centering
    \includegraphics[width=\textwidth]{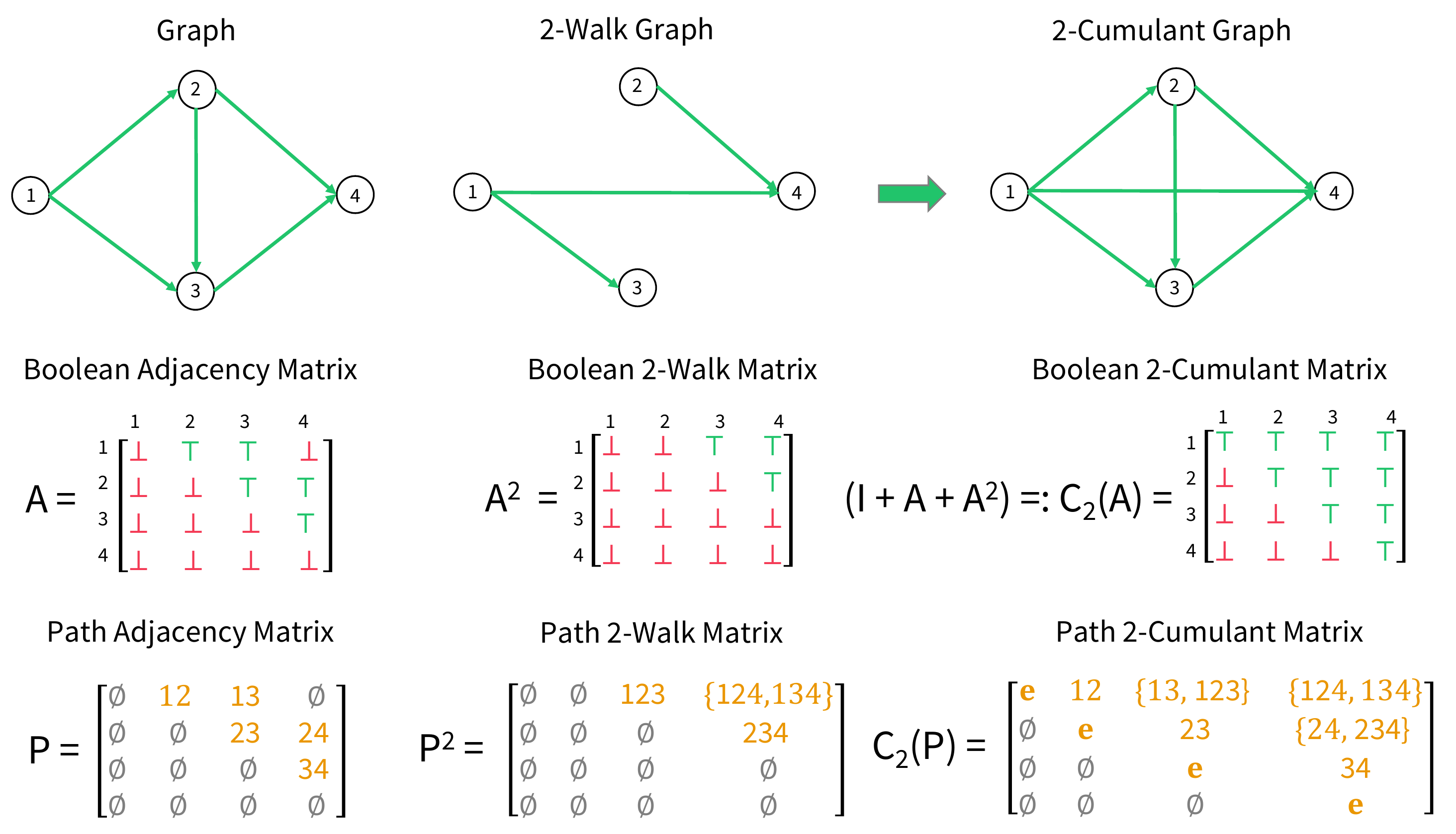}
    \caption{Simple directed graphs are equivalent to matrices valued in the Boolean semi-ring, \Cref{defn:boolean-semi-ring}. 
    % Higher 
    Taking powers of this matrix reveals longer walks between nodes. The $k$-cumulant aggregates walks of length $k$ or less. Using a matrix with the path semi-ring, \Cref{defn:path-semi-ring}, encodes different paths between the same pair of nodes.}
    \label{fig:graph-cartoons}
\end{figure}

%\jacob{Should change the diagonal in the above figure to "1, 2, 3, and 4" respectively per discussion with Billy and Justin}

Interpretations of Boolean matrix addition and multiplication in graph theory are classical and well-understood. If $A$ and $B$ are two $n\times n$ Boolean matrices---equivalently viewed as graphs $G$ and $H$---then $A+B$ can be interpreted as the union of the edge sets of $G$ and $H$.
The matrix product $AB$ can be viewed as the matrix of length 2 walks, where the first edge is traversed in $G$ and the second edge is traversed in $H$.
When $A=B$, the matrix $A^2$ has a $\top$ in entry $(i,j)$ if and only if there is a length 2 walk from $i$ to $j$.
However, if there are multiple length 2 walks between $i$ and $j$, the matrix $A^2$ cannot detect that. For this purpose it is better to use adjacency matrices valued in the natural numbers $\mathbb{N}$ or even $\Path(G)$ as these will encode the number and exact names of routes between nodes, respectively; see \Cref{fig:graph-cartoons}.

The rest of this section is devoted to understanding the higher powers of a matrix TVG $M$.
As we will see, the matrix $M^k$ will encode the intervals of time in which an instantaneous length $k$ walk exists between two nodes.
This will require proving an isomorphism between matrix TVGs and functions from $\R$ to $\Mat_n(\Bool)$. 
This in turn depends on the following ``snapshot'' construction.

\begin{definition}[Snapshot of a matrix TVG, cf. \cite{myers2023temporal}]\label{def:snapshot}
    Let $tI$ be the matrix with $(tI)_{ii}=\{t\}$ along the diagonal and $(tI)_{ij} = \varnothing$ for off-diagonal entries.
    If $M\in \Mat_n(\mathcal{P}(\R))$ is a matrix TVG, then the \define{snapshot} of $M$ \define{at $t$} is the matrix $(tIM)$, viewed as matrix of Booleans $\snap_t(M)$, i.e.,
    \[
        \snap_t(M)_{ij}=\top \iff  (tIM)_{ij}=\{t\} \quad \text{and} \quad \snap_t(M)_{ij}=\bot \iff (tIM)_{ij}=\varnothing.
    \]
    This defines a semi-ring homomorphism $\snap_t : \Mat_n(\powR) \to \Mat_n(\Bool)$.
\end{definition}

The following theorem is fundamental to our paper. Its proof is somewhat lengthy and deferred to the Appendix, where it appears under \Cref{thm:boolean-curry-appendix}.

\begin{theorem}\label{thm:boolean-curry}
The semi-ring of matrix TVGs $\,\Mat_n(\powR)$ is isomorphic to the semi-ring of functions from $\R$ to $\Mat_n(\Bool)$.
This isomorphism is witnessed by the homomorphism
\[
    \Psi: \Mat_n(\powR) \to \Fun(\R,\Mat_n(\Bool)), \quad \text{where} \quad \Psi(M)(t)=\snap_t(M)
\]
is the snapshot of $M$ at time $t$.
\end{theorem}

Although \Cref{thm:boolean-curry} allows us to interpret a matrix TVG as a family of $\R$-indexed Boolean matrices, for computational purposes it is better to work intrinsically with matrices of lifetimes, where each entry is a finite union of closed intervals.
To wit, if $M$ is a matrix TVG with at most $L$ closed intervals in each entry, then finding $M_{ij}\cap M_{jk}$ requires at most $L^2$ intersection operations, but working with the Boolean perspective, i.e., $(m_{ij} \wedge m_{jk})(t)$, requires checking every $t\in \R$ to see if the resulting matrix entry evaluates true at $t$---a computational impossibility.
As such, we finish this section with a few more TVG-centric constructions and conclude with a corollary that interprets these results through the Boolean lens of \Cref{thm:boolean-curry}.

\begin{definition}[Lifetime of a Walk]\label{def:walk-lifetime}
    Assume  $\cG = (G, \ell_M)$ is a TVG, as in \Cref{defn:TVG}, and $\gamma=[v_{i_0}, \ldots, v_{i_k}]$ is a list of $k+1$ nodes defining a $k$-walk (length $k$ path) in $G$.
    The \define{walk lifetime} is the intersection of the lifetimes of each edge appearing in the walk, i.e.,
    \[
        \ell_M(\gamma)=\ell_M([v_{i_0},v_{i_1}])\cdot \ell_M([v_{i_1},v_{i_2}])\cdots \ell_M([v_{i_{k-1}},v_{i_k}]).
    \]
    For a walk of length 0, i.e., $\gamma=[v_i]$, then $\ell_M([v_i])=\ell_M(v_i)$ is the lifetime of that single vertex.
\end{definition}

\Cref{def:walk-lifetime} connects \Cref{defn:path-semi-ring} with \Cref{def:lifetime-semi-ring} as follows:

\begin{lemma}
    Assume $\cG = (G, \ell_M)$ is a TVG, but where $\ell_M(v_i)=\R$ for every vertex $v_i$, then
    the lifetime function $\ell_M$ extends to a semi-ring homomorphism (\Cref{def:semi-ring-homomorphism})
    \[
    \ell_M : \Path(G) \to \subR, \quad \text{where} \quad \sum_i \gamma_i \mapsto \bigcup_i \ell_M(\gamma_i).
    \] 
\end{lemma}
\begin{proof}
    The assumption that each vertex lives forever, i.e., $\ell_M(v_i)=\R$, guarantees that \linebreak $\ell_M(1)=\ell_M(\sum [v_i])=\R$. The additive and multiplicative properties then hold by definition.  
\end{proof}

\begin{corollary}[Higher Powers of a Matrix TVG]\label{cor:union-walk-lifetimes}
If $M\in \Mat_{n}(\mathcal{P}(\R))$ is a matrix TVG, then $M^k_{ij}$ is the union of lifetimes over which there exists a length $k$ walk from node $i$ to node $j$.
\end{corollary}

\begin{figure}
    \centering
    \includegraphics[width=\textwidth]{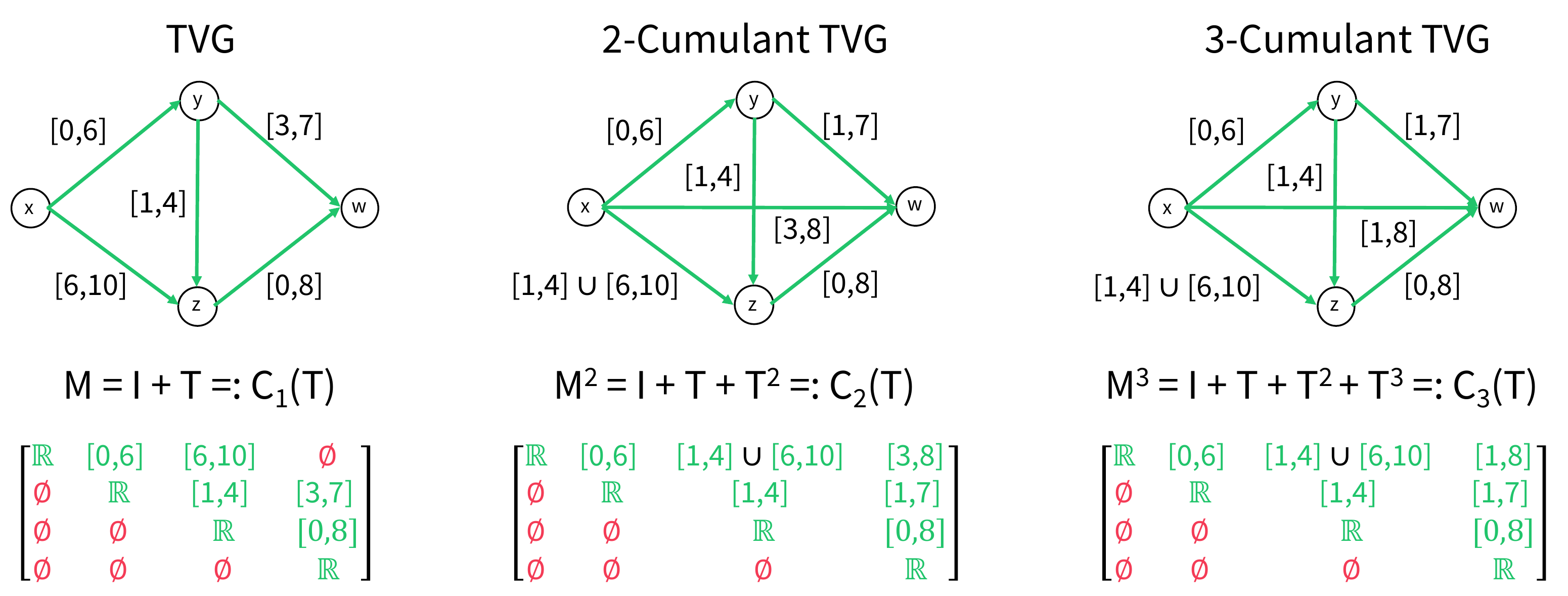}
    \caption{A time-varying graph can be viewed as a matrix valued in the lifetime semi-ring, \Cref{def:lifetime-semi-ring}. Higher powers of this matrix model times during which longer walks can occur. The $k$-cumulant aggregates windows of opportunity over which walks of length $k$ or less can occur. Notice that the $k$-cumulant, for increasing $k$, defines a filtration of lifetimes over each edge.}
    \label{fig:TVG-cartoons}
\end{figure}

\subsection{Measuring Communication Capacity of STARLINK via the Kleene Star}\label{sec:lifetime-curves}

As \Cref{thm:boolean-curry} and Corollary \Cref{cor:union-walk-lifetimes} indicate, higher powers of a matrix TVG model when in time instantaneous walks can occur between nodes.
In this section we show how the Kleene star and its truncation, called the cumulant, leads to a particular growth filtration of the temporal capacity for communicating between two nodes in a TVG.
This provides a summary statistic for measuring proximity of a TVG to a strongly connected one, which we illustrate using increasing size samples from STARLINK---a SpaceX-operated internet service that features over 3,000 satellites in low Earth orbit.
Our findings illustrate that temporal capacity illustrates a phase transition around 30 satellites, with larger subnetworks being strongly connected in a sense defined here.

\begin{definition}[Kleene Star and k-Cumulant] \label{def:kleene_star}
Let $A$ be an arbitrary matrix TVG.
The \define{Kleene star} of $A$, written
\[
A^*:=I + A + A^2 + A^3 + \cdots,
\]
is the matrix whose $(i,j)$ entry is the subset of $\
R$ where any communication from node $i$ to node $j$ can occur, perhaps using a walk of arbitrary length.
The \define{k-cumulant} is the sum of the first $k+1$ terms in the Kleene star and is written $C_k(A)$ or $C_k$, when the matrix $A$ is clear from context.
Finally, we say the \define{Kleene star converges} at radius $r$ if $A^*=C_r(A)$ for $r\in\mathbb{N}$.
See \Cref{fig:TVG-cartoons} for an example TVG and its cumulants, whose Kleene star converges at $r=3$.
\end{definition}

\begin{figure}[h]
\begin{center}
\begin{tikzpicture}
% https://tex.stackexchange.com/questions/136427/problems-using-addlegendentry
% https://tex.stackexchange.com/questions/591157/insert-legend-in-a-looping
\begin{axis}[
    width=12cm,
    height=7.5cm,
    title={Aggregate Average $L(\mu^k)(i, j)$ \\ % $ = $ sum of intervals in entry $(i, j)$ \\
        {\footnotesize $(30)$ $n$-Node Simulations}},
    xlabel={$k$},
    ylabel={Time in Seconds},
    ymin=0,
    ymax=90000,
]
\foreach \yindex/\nval in {4/20,6/30,8/40} {
    \edef\temp{\noexpand\addlegendentry{$n=\nval$}}
    % \addlegendentry{$n=\nval$};
    \addplot+[ultra thick] table[y index = \yindex]
        {simulations/average_lifetime_matrix_aggregates_today.dat};
    \temp
}
\foreach \yindex/\nval in {4/50,5/70,6/100} {
    \edef\temp{\noexpand\addlegendentry{$n=\nval$}}
    % \addlegendentry{$n=\nval$};
    \addplot+[ultra thick] table[y index = \yindex]
        {simulations/average_lifetime_matrix_aggregates_today_ext.dat};
    \temp
}
\end{axis}
\end{tikzpicture}
\end{center}

\caption{Samples of $n$ nodes from STARLINK are simulated using SOAP for one day or 86,400 seconds. For each of these simulations, the average lifetime curve (defined in \Cref{lem:lifetime-curves}) is computed across all nodes $i,j$ for each power $k$ of the TVG matrix. There appears to be a jump in average connectivity above $n=30$ nodes.}
\label{fig:lifetime-curves-sample}
\end{figure}
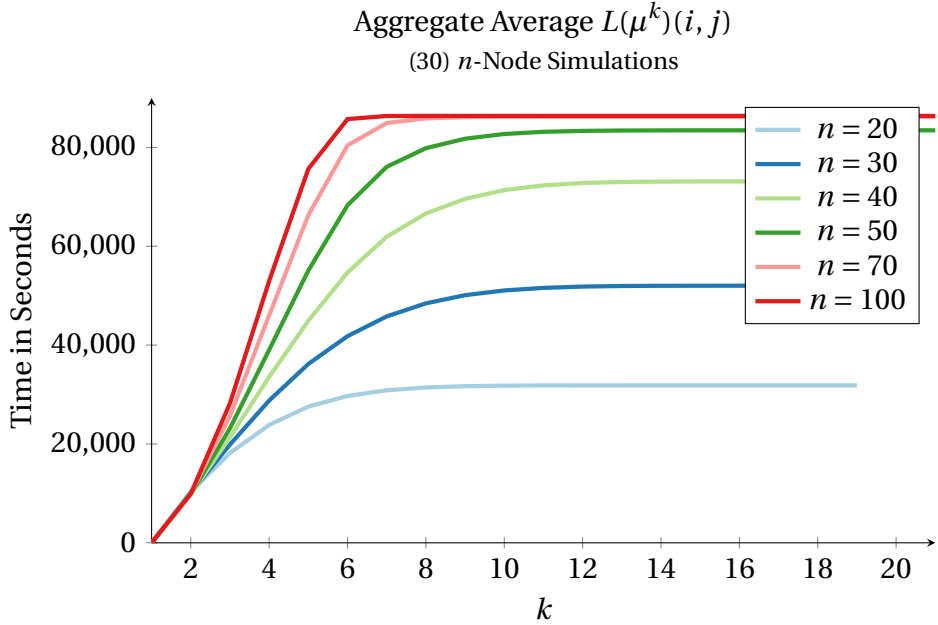

% \begin{figure}[h]
%     \centering
%     \includegraphics[width=.6\textwidth]{aggregate-lifetime.png}
%     \caption{Samples of $n$ nodes from STARLINK are simulated using SOAP for one day or 86,400 seconds. For each of these simulations, the average lifetime curve (defined in \Cref{lem:lifetime-curves}) is computed across all nodes $i,j$ for each power $k$ of the TVG matrix. There appears to be a jump in average connectivity above $n=30$ nodes.}
%     \label{fig:lifetime-curves-sample}
% \end{figure}

\begin{remark}\label{rmk:matrix-powers}
    The Kleene star construction helps illustrate why we separated out the notion of the lifetime matrix $M$ (\Cref{defn:matrix-TVG}) associated to a simple TVG $\cG = (G, \ell_M)$ from its adjacency matrix (\Cref{defn:adjacency-matrix-TVG}).
    In particular, if $\ell_M(v)=\R$, then $M=I + A$ and we can use the idempotency of $\powR$ (\Cref{rmk:idempotent}) to show that
    \[
        C_k(A)=I + A + \cdots + A^k = (I+A)^k = M^k.
    \]
    To see this, note that $A+A=A$ and
    \[
    M^2 = (I+A)^2=(I+A)(I+A)=I^2+IA+AI+A^2=I+A+A^2 = C_2(A).
    \]
    The general result then follows by induction.
\end{remark}

Since addition of matrix TVGs is the union of lifetimes, the $k$-cumulant defines over each edge a chain of nested subsets of $\R$.

\begin{lemma}[Lifetime Curves]\label{lem:lifetime-curves}
    Suppose $\cG = (G, \ell_M)$ is a simple TVG with $\ell_M(v)=\R$, so that the corresponding matrix of lifetimes satisfies $M=I+A$.
    %\jacob{For such a TVG? To sounds weird here}
    For such a TVG we have a \define{lifetime filtration}, which associates to each edge $[i,j]$ of $G$ the chain of subsets
    \[
    M^0_{ij}=\varnothing \subseteq M^1_{ij} \subseteq \cdots \subseteq M^k_{ij} \subseteq \cdots,
    \]
    whose $k$th entry is the union of lifetimes of walks length $k$ or less; see \Cref{def:walk-lifetime}.
    Moreover, since closed intervals in $\R$ are Lebesgue measurable, we can associate to each edge $[i,j]$ of ${\cG = (G, \ell_M)}$ its \define{lifetime curve}, which is a non-decreasing function
    \[
    L(M^{\bullet})(i,j): \mathbb{N} \to \R, \quad \text{where} \quad L(\mu^k)(i,j) = \mu(M^k_{ij})
    \]
    is the sum of the lengths of intervals in entry $(i,j)$ of $M^k$. This is non-infinite, assuming all off-diagonal entries of $M$ have finite measure. See \Cref{fig:lifetime-curves-sample} for some examples.
\end{lemma}

The lifetime curves are meant to provide a summary statistic for measuring how close a TVG is to a strongly connected one, which we define now.

\begin{definition}[Strongly Connected TVGs]
A matrix TVG $M\in \Mat_n(\powR)$ is \define{strongly connected} if $M^*=I+M+M^2+\cdots$ is equal to the constant matrix where every entry has value $\R$.
More generally, if $W\subseteq \R$ is some subset of $\R$, then $M\in \Mat_n(\mathcal{P}(W))$ is \define{strongly connected over $W$} if $M^*$ is the constant matrix with value $W$. \footnote{Typically, we will assume $W=[0,86400]$ for simulations that occur over one day, measured in seconds.}
\end{definition}

\begin{remark}[Strong Connectivity for Directed Graphs]
    As a reminder, a finite directed graph is strongly connected if one can go from any node to any other node using a directed path. This is equivalent to the Kleene star being the constant $\top$ matrix, which justifies our definition above.
\end{remark}

%\todo{Add a definition that says that a TVG is strongly connected if its Kleene star is the all $\R$ matrix.}
% A simple TVG $\cG = (G, \ell)$ is \define{strongly connected} if its Kleene star matrix consists of only $\R$ entries.

% Carr\'e establishes that for $n\geq |V|$ the truncated Kleene star of a semi-definite TVG stabilizes, i.e. $A^*_n = A^*_{n+1}$.
% Our first result is a sharpening of this bound.

As \Cref{fig:lifetime-curves-sample} indicates, the lifetime curves of a satellite system sampled from STARLINK can illustrate radically different behaviors depending on the number of nodes in the system.
The systems here with 40 or more nodes seem to illustrate almost exponential growth in the average lifetime (averaged over all pairs of nodes) as a function of the walk-length $k$, whereas systems with 30 or fewer nodes have growth more akin to a square root function.
However, all curves reach a horizontal asymptote, as the Kleene star converges for finite values of $r$.
We review this result in the most general setting, which was first established by Carr\'e as \cite[Thm. 3.1]{Carre1971} and then re-formulated in Baras \cite[p. 19]{Baras2010}, before providing our own improvement on this result for the semi-ring $\powR$.

\begin{theorem}[Kleene Star Convergence, cf.  \cite{Carre1971} and \cite{Baras2010}]\label{thm:carre-convergence}
    Suppose $G=(V,E)$ is a simple directed graph that is \define{weighted} in a semi-ring $(S,\oplus,\otimes,\n,\e)$, i.e., a map $w:V\times V \to S$ is given. 
    If for every cycle, i.e., a path $\gamma=[v_{i_0},\ldots,v_{i_n}]$ where $v_{i_0}=v_{i_n}$, the weight
    \[
    w(\gamma)=w(v_{i_0},v_{i_1})\otimes \cdots \otimes w(v_{i_{n-1}},v_{i_n}) \quad \text{satisfies} \quad w(\gamma)\oplus \e = \e,
    \]
    then the Kleene star of the weighted adjacency matrix $A_{ij}=w(i,j)$ converges for $r\leq |V|-1$.
\end{theorem}

\begin{figure}[h]
\begin{center}
\begin{tikzpicture}
\begin{axis}[
    width=6cm,
    ybar interval,
    title={Histogram of Diameters},
    xlabel={Slice Diameter},
    ylabel={Slice Count},
    x tick label as interval=false,
    % xticklabel = \pgfmathprintnumber\tick--\pgfmathprintnumber\nexttick
    xtick={5, 6, 7}
]
\addplot+ [hist={data=x, bins=3}]
    file {simulations/starlink_2023_06_0_diameter_histogram.dat};
\end{axis}
\end{tikzpicture}
\quad
\begin{tikzpicture}
\begin{axis}[
    width=6cm,
    title={Graph of TVG Diameter},
    xlabel={Slice Time},
    ylabel={Slice Diameter},
]
\addplot [thick,blue] file {simulations/starlink_2023_06_0_diameter_histogram_even.dat};
\end{axis}
\end{tikzpicture}
\end{center}
\caption{A histogram of diameters from slices of a one-day STARLINK simulation with 100 nodes. The maximum observed diameter was 7, but the typical diameter was around 6.}
    \label{fig:diam-histogram}
\end{figure}
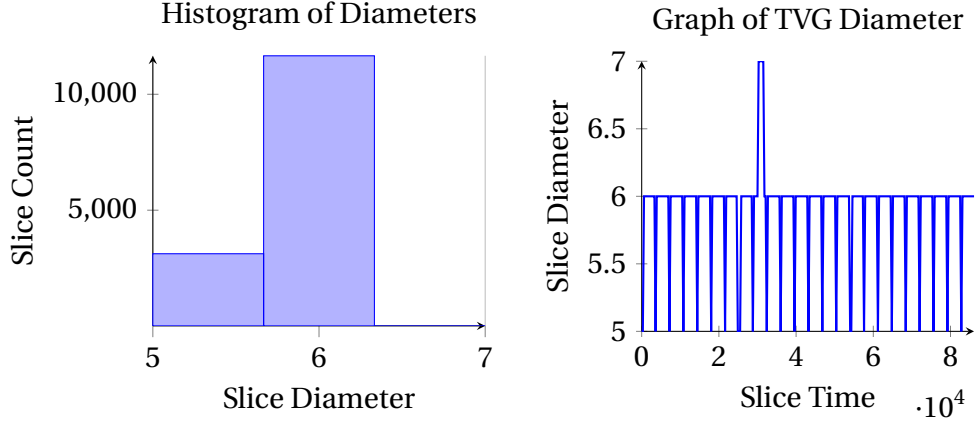

% \begin{figure}[h]
%     \centering
%     \includegraphics[width=.5\textwidth]{histogram-diams-80starlink-crop.png}
%     \caption{A histogram of diameters from slices of a one-day STARLINK simulation with 80 nodes. The maximum observed diameter was 15, but the typical diameter was around 9.}
%     \label{fig:diam-histogram}
% \end{figure}

The hypotheses of \Cref{thm:carre-convergence} holds for the lifetime and powerset semi-ring as every cycle has a lifetime $w(\gamma)$ such that $w(\gamma)\cup \R =\R$.
This means that the lifetime curves of \Cref{fig:lifetime-curves-sample} will provably stabilize at $k=n-1$.
Since some of our simulations take up to $n=100$ nodes in a single simulation, it is possible that the $k$-cumulant will need to be computed up to and including $r=99$.
However the next proposition, whose proof is deferred to the appendix under \Cref{prop:Kleene-star-diam-appendix}, and our simulation results indicate that the convergence radius of the Kleene star for STARLINK simulations is typically much smaller.

\begin{figure}[h]
\begin{center}
% Average $L^k(i, j)$ over $(30)$ $n$-Node Simulations \\
\vspace{5pt}
\begin{tikzpicture}[baseline]
\begin{axis}[
    width=3.5cm,
    height=5cm,
    title={$n = 10$},
    ylabel={Time in Seconds},
    xmin=1, 
    xmax=7,
    ymin=0,
    ymax=86400
    % y tick label style={/pgf/number format/.cd,fixed,precision=5,/tikz/.cd},
    % scaled y ticks = false,
    % axis lines=left,
    % cycle list/Paired
]
\foreach \yindex in {1,...,30} {
    \addplot+[thick] table[y index = \yindex]
        {simulations/average_lifetime_curve_10_sats_today.dat};
}
\end{axis}
\end{tikzpicture}
\quad
\begin{tikzpicture}[baseline]
\begin{axis}[
    width=3.5cm,
    height=5cm,
    title={$n = 20$},
    xmin=1, 
    xmax=12,
    ymin=0,
    ymax=86400,
    ytick=\empty
    % y tick label style={/pgf/number format/.cd,fixed,precision=5,/tikz/.cd},
    % scaled y ticks = false,
    % axis lines=left,
    % cycle list/Paired
]
\foreach \yindex in {1,...,30} {
    \addplot+[thick] table[y index = \yindex]
        {simulations/average_lifetime_curve_20_sats_today.dat};
}
\end{axis}
\end{tikzpicture}
\quad
\begin{tikzpicture}[baseline]
\begin{axis}[
    width=4cm,
    height=5cm,
    title={$n = 50$},
    xmin=1, 
    xmax=16,
    ymin=0,
    ymax=86400,
    ytick=\empty
    % y tick label style={/pgf/number format/.cd,fixed,precision=5,/tikz/.cd},
    % scaled y ticks = false,
    % axis lines=left,
    % cycle list/Paired
]
\foreach \yindex in {1,...,30} {
    \addplot+[thick] table[y index = \yindex]
        {simulations/average_lifetime_curve_50_sats_today.dat};
}
\end{axis}
\end{tikzpicture}
\quad
\begin{tikzpicture}[baseline]
\begin{axis}[
    width=4cm,
    height=5cm,
    title={$n = 100$},
    xmin=1, 
    xmax=16,
    ymin=0,
    ymax=87000,
    ytick=\empty
    % y tick label style={/pgf/number format/.cd,fixed,precision=5,/tikz/.cd},
    % scaled y ticks = false,
    % axis lines=left,
    % cycle list/Paired
]
\foreach \yindex in {1,...,30} {
    \addplot+[thick] table[y index = \yindex]
        {simulations/average_lifetime_curve_100_sats_today.dat};
}
\end{axis}
\end{tikzpicture}
\end{center}

\begin{center}
% Range of Two Standard Deviations over $(30)$ $n$-Node Simulations \\
\vspace{5pt}
\begin{tikzpicture}[baseline]
\begin{axis}[
    width=3.5cm,
    height=5cm,
    xlabel={$k$},
    ylabel={Time in Seconds},
    xmin=1, 
    xmax=7,
    ymin=0,
    ymax=87000
]
    \addplot[thick, name path=U] table[y index = 3]
        {simulations/range_sd_lifetime_matrix_10_sats.dat};
    \addplot[thick, name path=L] table[y index = 2]
        {simulations/range_sd_lifetime_matrix_10_sats.dat};
    \addplot[pattern=north west lines] fill between [of=U and L];
    \addplot[ultra thick] table[y index = 1]
        {simulations/range_sd_lifetime_matrix_10_sats.dat};
\end{axis}
\end{tikzpicture}
\quad
\begin{tikzpicture}[baseline]
\begin{axis}[
    width=3.5cm,
    height=5cm,
    xlabel={$k$},
    xmin=1, 
    xmax=12,
    ymin=0,
    ymax=87000,
    ytick=\empty
]
    \addplot[thick, name path=U] table[y index = 3]
        {simulations/range_sd_lifetime_matrix_20_sats.dat};
    \addplot[thick, name path=L] table[y index = 2]
        {simulations/range_sd_lifetime_matrix_20_sats.dat};
    \addplot[pattern=north west lines] fill between [of=U and L];
    \addplot[ultra thick] table[y index = 1]
        {simulations/range_sd_lifetime_matrix_20_sats.dat};
\end{axis}
\end{tikzpicture}
\quad
\begin{tikzpicture}[baseline]
\begin{axis}[
    width=4cm,
    height=5cm,
    xlabel={$k$},
    xmin=1, 
    xmax=16,
    ymin=0,
    ymax=87000,
    ytick=\empty
]
    \addplot[thick, name path=U] table[y index = 3]
        {simulations/range_sd_lifetime_matrix_50_sats.dat};
    \addplot[thick, name path=L] table[y index = 2]
        {simulations/range_sd_lifetime_matrix_50_sats.dat};
    \addplot[pattern=north west lines] fill between [of=U and L];
    \addplot[ultra thick] table[y index = 1]
        {simulations/range_sd_lifetime_matrix_50_sats.dat};
\end{axis}
\end{tikzpicture}
\quad
\begin{tikzpicture}[baseline]
\begin{axis}[
    width=4cm,
    height=5cm,
    xlabel={$k$},
    xmin=1, 
    xmax=16,
    ymin=0,
    ymax=87000,
    ytick=\empty
]
    \addplot[thick, name path=U] table[y index = 3]
        {simulations/range_sd_lifetime_matrix_100_sats.dat};
    \addplot[thick, name path=L] table[y index = 2]
        {simulations/range_sd_lifetime_matrix_100_sats.dat};
    \addplot[pattern=north west lines] fill between [of=U and L];
    \addplot[ultra thick] table[y index = 1]
        {simulations/range_sd_lifetime_matrix_100_sats.dat};
\end{axis}
\end{tikzpicture}
\end{center}
\caption{30 simulations of $n=10, 20, 50, 100$ (top) randomly sampled nodes from STARLINK over the course of one day, i.e., $86400$ seconds. Confidence intervals are shown below.}
    \label{fig:STARLINK-10-and-100}
\end{figure}
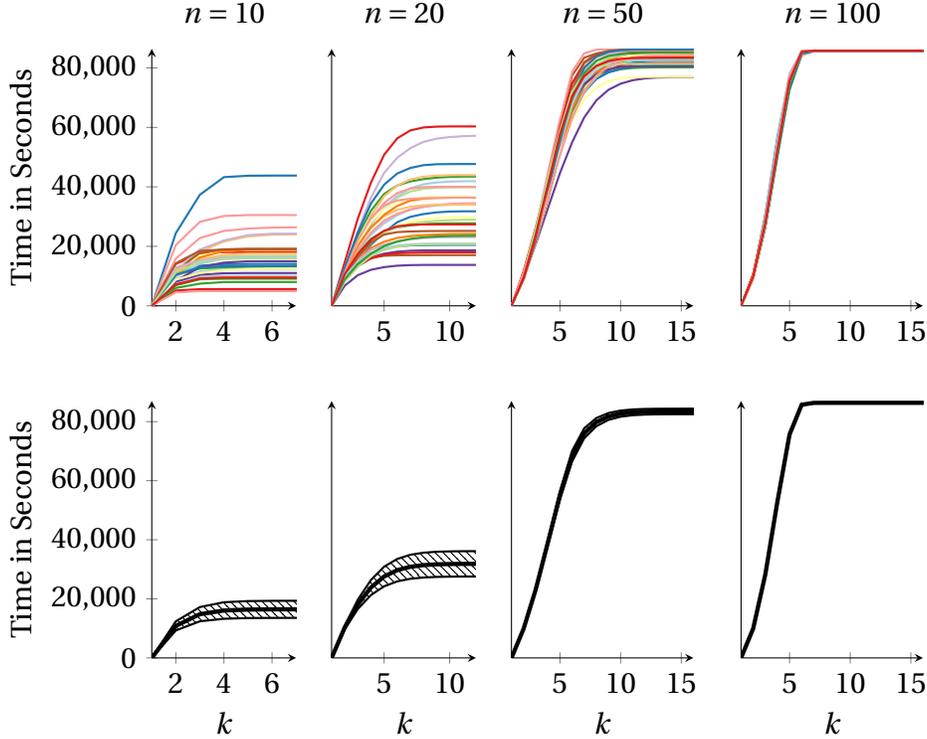

% \begin{figure}[h!]
%     \centering
%     \includegraphics[width=0.49\textwidth]{starlink_10_sats_30_overlay.png}
%     \includegraphics[width=0.49\textwidth]{starlink_100_sats_30_overlay.png}
%     \includegraphics[width=0.49\textwidth]{starlink_10_sats_30_std.png}
%     \includegraphics[width=0.49\textwidth]{starlink_100_sats_30_std.png}
%     \caption{30 simulations of $n=10$ (left) and $n=100$ (right) randomly sampled nodes from STARLINK over the course of one day, i.e.~$86400$ seconds. Confidence intervals are shown below.}
%     \label{fig:STARLINK-10-and-100}
% \end{figure}

\begin{proposition}[Convergence at the Temporal Diameter]\label{prop:Kleene-star-diam}
Let $A$ be the adjacency matrix for a simple TVG $\cG = (G, \ell_M)$.
The $(i,j)^{th}$ entry of the $k$-cumulant $C_k(A)$ stabilizes after $d_{ij}$, where
\[
d_{ij}= \max_{t \mid \exists \gamma:i\rightsquigarrow j} \min_{\gamma} |\gamma|
\]
is the length of the longest shortest path from $i$ to $j$, disregarding times where no such path exists. 
We set $d_{ii}=0$, by convention.
We define the \define{temporal diameter} of a TVG $\cG$ to be
\[
\diam (\cG)=\max_{ij} d_{ij}.
\]
Consequently, the Kleene star $A^*$ convergences for $r\geq \diam (\cG)$.
\end{proposition}
% \begin{proof}
% By \Cref{rmk:matrix-powers}, the question of whether or not an entry of $C_k(A)$ stabilizes at some value of $k$ is equivalent to asking if that entry of $M^k=(I+A)^k$ stabilizes at $k$.
% Moreover, by \Cref{thm:boolean-curry}, we can faithfully recover the values of $M^k$ by considering the associated Boolean function $m^k:V\times V\times \R \to \Bool$, when is $\Psi(M^k)=\Psi(M)^k$.
% This is the logical \texttt{OR} (union) of $a^p$, which is $\Psi(A^p)=\Psi(A)^p$, across all $p=1,\dots,k$ along with the identity relation, thus it suffices to characterize what nodes are related at time $t$ using $a^p(t):V\times V \to \Bool$.
% For $p=1$, this is simple: $a(i,j,t)=1$ if there is an edge from $i$ to $j$ at time $t$.
% By direct inspection of matrix multiplication, one can see that $a^2(i,j,t)=1$ if and only if there is a length-2 walk at time $t$.
% More generally, $a^p(i,j,t)=1$ if and only if there is a length-$p$ walk from $i$ to $j$ at time $t$. If $i$ and $j$ are connected by some walk, then $d_{ij}$ is defined to be the maximum over all $t\in\R$. 
% Since the maximum diameter of any connected component of a graph on $|V|$-vertices is $|V|-1$, and $a^p(i,j,\bullet):\R \to \Bool$ is piece-wise constant on finitely many pieces (by virtue of $\subR$ rather than $\powR$), this maximum necessarily exists.
% \end{proof}

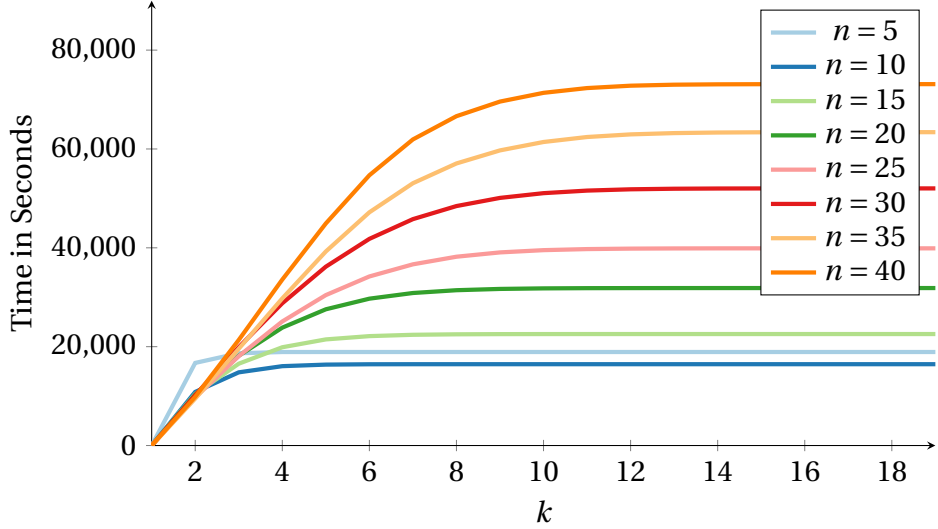
\begin{figure}[h]
\begin{center}
\begin{tikzpicture}
% https://tex.stackexchange.com/questions/136427/problems-using-addlegendentry
% https://tex.stackexchange.com/questions/591157/insert-legend-in-a-looping
\begin{axis}[
    width=12cm,
    height=7.5cm,
    title={Aggregate Average $L(\mu^k)(i, j) = $ sum of intervals in entry $(i, j)$ \\
        {\footnotesize $(30)$ $n$-Node Simulations}},
    xlabel={$k$},
    ylabel={Time in Seconds},
    ymin=0,
    ymax=90000,
]
\foreach \yindex/\nval in {1/5,2/10,3/15,4/20,5/25,6/30,7/35,8/40} {
    \edef\temp{\noexpand\addlegendentry{$n=\nval$}}
    % \addlegendentry{$n=\nval$};
    \addplot+[ultra thick] table[y index = \yindex]
        {simulations/average_lifetime_matrix_aggregates_today.dat};
    \temp
}
% \foreach \yindex/\nval in {4/50,5/70,6/100} {
%     \edef\temp{\noexpand\addlegendentry{$n=\nval$}}
%     % \addlegendentry{$n=\nval$};
%     \addplot+[ultra thick] table[y index = \yindex]
%         {simulations/average_lifetime_matrix_aggregates_today_ext.dat};
%     \temp
% }
\end{axis}
\end{tikzpicture}
\end{center}

    \caption{Average lifetime curves across 30 simulations for node systems ranging from $n=5$ to $n=40$, in increments of size five. Node systems larger than $n=30$ appear to be strongly connected over a one-day simulation, i.e., the Kleene star appears to converge to a constant matrix with value $86400$---the number of seconds in a day.}
    \label{fig:STARLINK-confidence-summary}
\end{figure}

% \begin{figure}[h]
%     \centering
%     \includegraphics[width=0.7\textwidth]{starlink-confidence-mean-crop.png}
%     \caption{Average lifetime curves across 30 simulations for node systems ranging from $n=5$ to $n=40$, in increments of size five. Node systems larger than $n=30$ appear to be strongly connected over a one-day simulation, i.e.~the Kleene star appears to converge to a constant matrix with value $86400$---the number of seconds in a day.}
%     \label{fig:STARLINK-confidence-summary}
% \end{figure}

% \todo{Now need TUNG to help write out the non-degeneracy condition for matrices with semi-ring entries. point out how this improves the naive upper bound, but then motivate the Propagation Delay semi-ring as an example of a system where the Kleene star does not stabilize. ``Winding and covering space phenomenon example''}
% \tung{The convergence of Kleene star was stated for network whose arc values are in communtative semi-rings. The Propagation Delay semi-ring is not communtative.}

The upshot of \Cref{prop:Kleene-star-diam} is that one need only understand the maximum diameter of snapshots of $M$ measured across all possible values of $t\in\R$.
Based on the simulation described in \Cref{fig:diam-histogram}, the temporal diameter for that 100 node STARLINK simulation is actually 7.

% \todo{For a given $n$, we generated thirty simulations consisting of
% $n$ randomly sampled starlink satellites and calculated the average $L(\mu^k)(i, j)$
% for each associated TVG, which are on the left side of the figure. We then took the 
% 95 \% confidence intervals and plotted the average mean and band on the 
% right side of the figure. The bottom figure shows the confidence interval means for a 
% specified collection of $n$, showing that once you get arround thirty satellites, you have
% full coverage after about four steps of the truncated Kleene star.}

% Robert : removed 2023-08-29
However, further simulation indicates that for TVGs coming from STARLINK, the convergence radius for the Kleene star is actually much smaller than $r=7$, which would be the prediction of \Cref{prop:Kleene-star-diam} and \Cref{fig:diam-histogram}.
By inspecting \Cref{fig:STARLINK-10-and-100}, one can gleam convergence radii as a function of the number of nodes---this is indicated by where the lifetime curves plateau.
For more nodes, e.g., $n=150$ (not shown), this convergence radius appears to be even smaller than $r=7$ and is conjecturally less than $5$ for STARLINK systems with more than 150 nodes.
From purely geometric considerations, depicted in \Cref{fig:starlink-conjectured-radius}, the limiting radius should be $r=3$, but our experiments have not realized this limit.
As such we state weak and strong versions of our conjecture.
% Robert : removed 2023-08-29

\begin{figure}[h]
    \centering
    \includegraphics[width=0.6\textwidth]{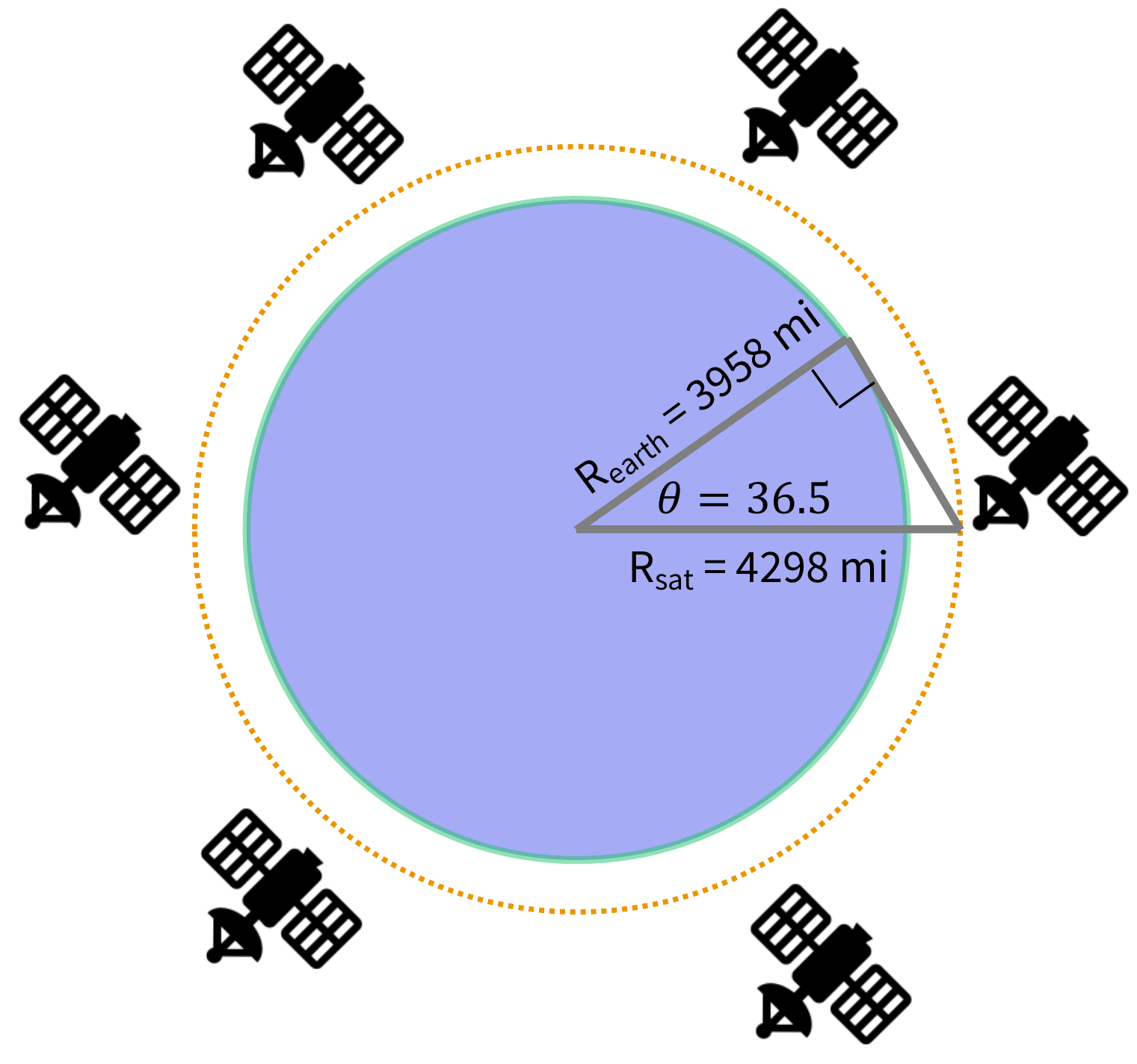}
    \caption{Starlink satellites typically orbit the Earth at an altitude of $\sim 340$ miles. Using line of sight considerations, each satellite should have approximately a $73^{\circ}$ view angle, which implies that any two antipodal satellites can be connected by a walk of length at most 3, assuming a sufficient number of satellites.}
    \label{fig:starlink-conjectured-radius}
\end{figure}

% Robert : removed 2023-08-29
\begin{conjecture}[Convergence Radius for STARLINK TVGs]\label{conjecture:r=3}
\begin{enumerate}
    \item[A.] (Weak Version) With high probability as $n\to \infty$ a randomly sampled sub-TVG of STARLINK will have a Kleene star that converges for $r \leq 5$.
    \item[B.] (Strong Version) With high probability as $n\to \infty$ a randomly sampled sub-TVG of STARLINK will have a Kleene star that converges at $r=3$.
\end{enumerate}
\end{conjecture}
% Robert : removed 2023-08-29

Either of these conjectures can be viewed as asserting properties about the 1-skeleton of the Vietoris-Rips complex of $n$ nodes sampled uniformly from the unit sphere with connectivity radius $r=\pi/3$.
Although in theory a length three path should connect any two points on the sphere, a discrete sample may require an additional two links: one up to the diameter three sub-skeleton and one down to a receiving node.

\subsection{Non-Convergence of the Kleene Star: Semi-Rings for Propagation Delay}\label{sec:propagation-delay}

In all of the previous sections, we have focused on TVGs coming from space networking scenarios, such as STARLINK, where the speed of light is negligible.
In what follows we are going to investigate satellite systems with assets around the Earth's moon (Luna) and Mars.
In both of these settings, line-of-sight communication may require significant time due to propagation delay, even though messages are transmitted at the speed of light.
As a reminder, one-way light travel to the Moon takes about 1.2 seconds; this represents a round trip time (RTT) that begins to preclude standard or traditional feedback mechanisms for reliability in communications though reactivity is still possible.
Messages sent to Mars can take anywhere between 3 and 22 minutes, depending on the relative location of Earth and Mars in their respective orbits. When the RTT is measured in minutes, reliability in communications tends to be based on being proactive rather than reactive.
All of this necessitates a new semi-ring capable of modelling propagation delay.

In this section we introduce the Propagation Delay semi-ring, which is a combination of the Lifetime semi-ring of \Cref{def:lifetime-semi-ring} with a delay parameter.
For this semi-ring \Cref{thm:boolean-curry} fails to hold, which further illustrates why considering time-varying graphs as an $\R$-indexed family of simple directed graphs (or a graph sequence, for that matter) fails to capture systems with non-trivial propagation delay.
Moreover, the classical Kleene star convergence result of Carr\'e (Theorem  \Cref{thm:carre-convergence}) does not apply to the Propagation Delay Semi-Ring.
This is illustrated experimentally using a 17-node simulation of an Earth-Mars-Moon system, where a ping originating at one node ``echoes'' ad infinitum throughout the network.
This lack of convergence of the Kleene star may be interpreted as an algebraic characterization of the difficulty of deep space routing and communication.

%describe the problem
The starting point for defining the Propagation Delay Semi-Ring is based on the observation that the collection of additive endomorphisms of a semi-ring $S$ forms a semi-ring as well; see \Cref{lem:end-semi-ring}.
If one---necessarily---takes a relativistic perspective on space communication, then a message transcribed by one observer during an interval of time $I \subseteq \R$, may communicate this message, only to have it received later during the time interval 
\[
\varphi^{\epsilon}(I):=I^{\epsilon}=\{x+\epsilon \mid x\in I\}.
\] 
Composing such delay operators leads to another semi-ring structure, which we now specify.

%Let $\cL(\R)$ be the subset of real-line, we consider the 

%We are interested in a sub-semi-ring of $\End(\cL(\R))$ desribed as follows

 \begin{definition}[Propagation Delay Semi-Ring] 
 %For each $\epsilon \geq 0$ and $I \in \cL(\R)$, we consider an endomorphism $\varphi_I^{\epsilon}: \cL(\R) \to \cL(\R)$, defined as
%     \[ \varphi_I^{\epsilon}(K)=\{x+\epsilon \mid x\in I\cap K\}.  \]
%     %\[ \varphi_I^{\epsilon}(K)=\{x \mid x\in I\cap (K + \epsilon)\}.  \]
%     The collection of $\big\{\varphi^{\epsilon}_I| I \in \cL(\R), \epsilon \in \R^+ \big\}$ together with the operations $+$ and $\odot$ described above forms a semi-ring structure, called \emph{Propagation Delay semi-ring}.
Consider the set $\subR \times [0,\infty)$ of lifetimes (\Cref{def:lifetime-poset}) along with possible delays.
This set can be equipped with addition and multiplication operations, as follows:
\begin{itemize}
    \item $(I,s) + (J, t) := (I \cup J, \max\{s, t\})$, and
    \item $(I,s) \otimes (J, t) := (I \cap (J - s) , s + t)$.
\end{itemize}
The neutral elements are $\n=(\varnothing,0)$ and $\e=(\R,0)$ respectively.
\end{definition}

%Claim:
% We note that the above semi-ring is isomorphic to a simpler semi-ring on the set $\cL(\R)\times \R^+$ whose two operations are defined as: for a pair of $(I,t)$ and $(J, s)$ in $\cL(\R)\times \R^+$,

\begin{figure}[h!]
\centering
\begin{tikzpicture}[->,auto,node distance=3cm,
      thick,main node/.style={circle,draw,font=\sffamily\Large\bfseries},scale=1.5]
      %Author: Jacob Cleveland
    
      \node[main node] (1) at (-2,0) {$A$};
      \node[main node] (2) at (0,1) {$B$};
      \node[main node] (3) at (0,-1) {$C$};
      \node[main node] (4) at (2,0) {$D$};
    
      \path[every node/.style={font=\sffamily\small}]
          (1) edge[black] node [above=3mm] {([0,10]),1)} (2)
          (1) edge[black] node [below=3mm] {([0,10],3)} (3)
          (2) edge[black] node [above=3mm] {([9,15],3)} (4)
          (3) edge[black] node [below=3mm] {([9,10],2)} (4);
    \end{tikzpicture}
    \caption{Time-varying network of 4 nodes: each edge represents a connection, that is decorated with available time and delay time.}
    \label{fig:pds-cartoon}
\end{figure}
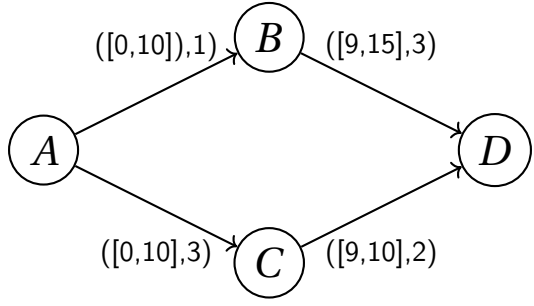

%Of course we can compose these two endomorphisms to obtain the \define{intersect and shift} endomorphism
%\[
%\varphi_I^{\epsilon}(J)=\{x+\epsilon \mid x\in I\cap J\}.
%\]
%\end{definition}

%\todo{Do a thorough example of a graph with different propagation delays on each edge. Integrate Brian's notes here.}

%\tung{the algebra now works out. } 
% \justin{NOTA BENE: Start with the observation that we change our matrices to be endomorphism-valued rather than subset-valued, but this is perhaps too general for our application purposes. Consequently, we will work with a sub-semi-ring of the endomorphism semi-ring, by looking at $\subR\times \R^+$, which is generated by intersect and shift operations, defined above. This is Brian's semi-ring, but will be known as the Propagation Delay semi-ring (PDS).}

\begin{remark}[Three Semi-Rings and Their Relationships]
The lifetime semi-ring (\Cref{def:lifetime-semi-ring}) injects into the propagation delay semi-ring, which in turn injects into the endomorphism semi-ring of the lifetime semi-ring $\End(\subR)$, defined generally in \Cref{lem:end-semi-ring}.
To see this last injection, 
note that each $I \in \subR$ and delay $s \geq 0$ and, we consider an ``intersect and shift'' endomorphism $\varphi_I^{s}: \cL(\R) \to \cL(\R)$, defined as
    \[ \varphi_I^{s}(K)=\{x+s \mid x\in I\cap K\}.  \]
Under this assignment the neutral element $\n=(\varnothing,0)$ is sent to the ``everything maps to $\varnothing$'' operator and $\e=(\R,0)$ is sent to the identity transformation.
\end{remark}

\begin{example}[4 Node Cartoon Example]\label{ex:2-node-PDS}
    In \Cref{fig:pds-cartoon} we consider an augmentation of a TVG with delays.
    The path $A\to B \to D$ leads to a composite
    \[
    ([0,10],1)\otimes ([9,15],3) = ([0,10]\cap [9-1,15-1],1+3)=([8,10],4).
    \]
    The path $A\to C \to D$ leads to a composite
    \[
    ([0,10],3)\otimes ([9,10],2) = ([0,10]\cap [9-3,10-3],3+2)=([6,7],5).
    \]
    The entry associated to the 2-walk matrix in the entry corresponding to $A\to D$ communications is $A^2_{14}$, whose value is then
    \[
        ([8,10],4) + ([6,7],5) = ([6,7]\cup [8,10],5).
    \]
    As this example shows, the addition operation assumes a worst-case scenario where---perhaps---a message is fragmented and sent along the two different routes and an operator at node D needs to wait till both messages arrive to re-assemble/decode the message. This makes implicit assumptions on the storage capacity at the node D, which we assume is arbitrarily large.
\end{example}

\begin{figure}[h]
    \centering
    \includegraphics[width=\textwidth]{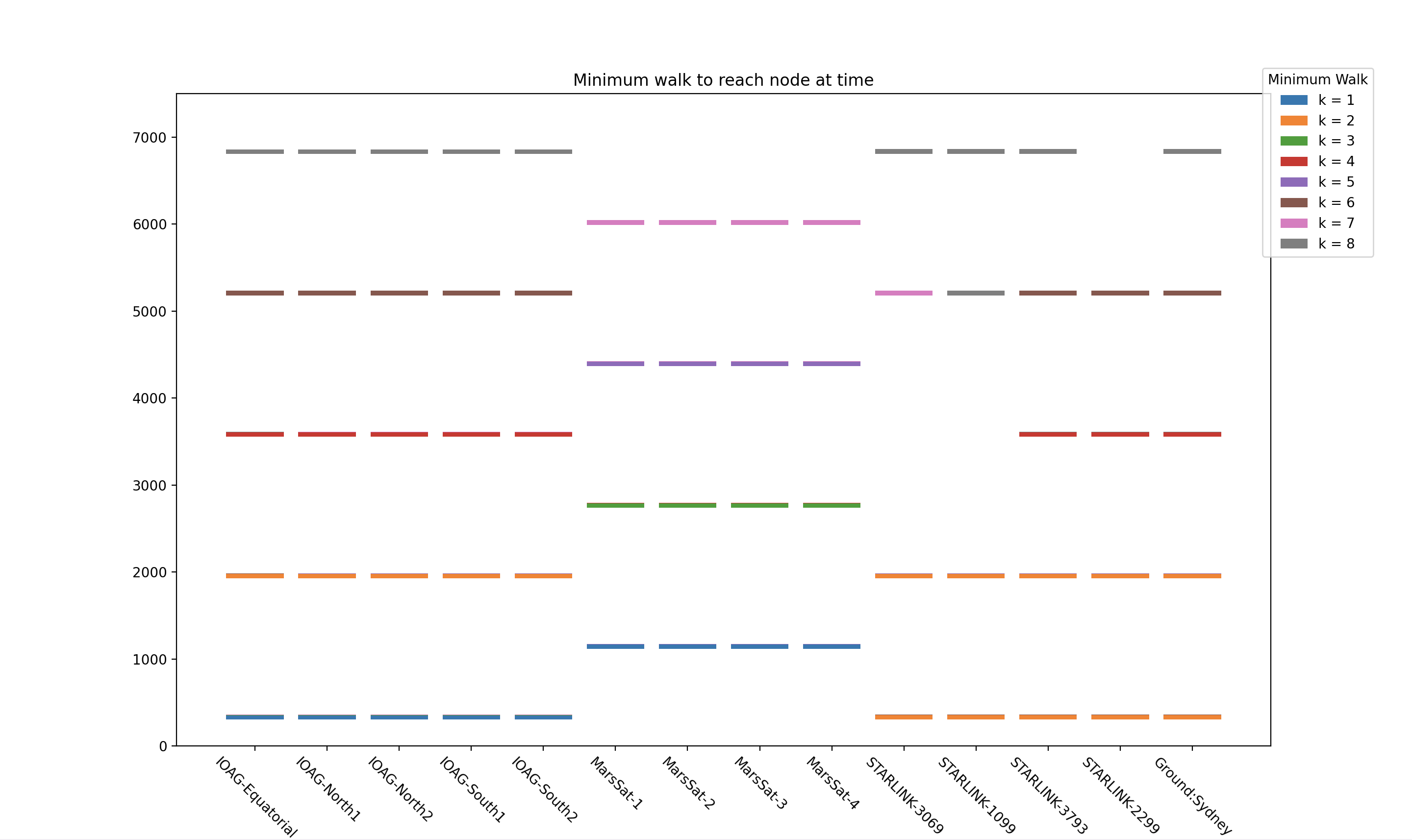}
    \caption{A ping is emitted at time $t=300$ from a ground station in Sydney. IOAG-recommended Lunar satellites receive the message less than 2 seconds later. This message is immediately relayed along all possible links, back to the Earth and on to Mars. Both Syndey and STARLINK satellites receive this repeated message as being the end of a length 2 walk. 
    %which is indicated by the orange bars. 
    Martian satelites receive the message from Sydney about 800 seconds later, at $t\approx 1100$, which broadcasts the message back to IOAG and STARLINK satellites. Messages are forwarded ad infinitum across the network, with walk length of the route indicated by color. This illustrates how the Kleene star fails to converge for the propagation delay semi-ring.}
    \label{fig:sydney-ping}
\end{figure}

Finally, we illustrate a simulated example using 14 assets:
\begin{itemize}
    \item One ground station in Sydney, Australia, which is the source of the ``ping'' at $t=300$.
    \item Four STARLINK satellites, which are out of view when the ping is first transmitted.
    \item Five (IOAG) Lunar satellites receive Sydney's ping less than 2 seconds later.
    \item Four Martian satellites receive the message between 13 and 14 minutes later.
\end{itemize}
Whenever an asset receive the ping, it automatically repeats the ping to all connected assets, which are received later according to the number of light seconds separating each asset.
STARLINK satellites move in and out of view and are connected by different length walks as the simulation evolves.
This sort of ``bent pipe'' communication is modelled via row-vector multiplication
\[
v^t(I+ A + A^2 + A^3 + \cdots + A^8 + \cdots)
\]
with the vector $v$ have a singleton set $\{300\}$ in the Sydney entry and $\varnothing$ in every other entry.
The vector of arrival times is depicted in \Cref{fig:sydney-ping}, color-coded by the length of the walk of the route traversed in the walk matrix $A^k$.

\subsection{Review of Prior Semi-Rings for Graph Optimization Problems}\label{sec:semi-ring-review-graph-optimization}

In this section we review some prior work on semi-rings for solving the all-pairs shortest-path problem for a weighted graph and the shortest path problem with time-inhomogeneous edges, with and without capacity constraints.

\subsubsection{Tropical/Min-Plus Semi-Ring}\label{sec:tropical}

As mentioned in \Cref{ex:2-node-PDS}, the propagation delay semi-ring uses a ``max-plus'' semi-ring in the second coordinate.
This semi-ring is dual to a more popular semi-ring for graph optimization problems.

\begin{definition}[Tropical Semi-Ring]\label{def:tropical}
Consider the set $\mathbb{T}=\mathbb{R}\cup\{\infty\}$. If we define $a\oplus b=\min\{a,b\}$ and $a\odot b=a+b$, then the set $\mathbb{T}$ with $\n=\infty$ and $\e=0$ defines the \define{tropical} or \define{min-plus semi-ring}.   
\end{definition}

\begin{remark}[All Pairs Shortest Path Problem]\label{rmk:shortest-path-Kleene-star}
    If $G$ is a simple directed graph with lengths asigned to each edge, then viewing this graph as weighted in $\mathbb{T}$ and calculating the Kleene star of this weighted adjacency matrix solves the \define{all pairs shortest path problem}, i.e., $A^*_{ij}$ will have the length of the shortest path from $i$ to $j$ in this entry.
    We note that if a cycle has negative weights, then the hypotheses of \Cref{thm:carre-convergence} fail to hold, as one can create an arbitrarily short (negative) length path by traversing this cycle repeatedly.
\end{remark}

%Important for realistic networking applications is the imposition of 
\subsubsection{Tropical Endomorphism Semi-Ring for Time-Varying Networks}\label{sec:tropical-endo}

As \Cref{rmk:shortest-path-Kleene-star} demonstrates, the shortest path problem can be solved by computing the Kleene star of the adjacency matrix weighted in the tropical semi-ring $\mathbb{T}$.
A model which generalizes this perspective, but also allows time-varying edge lengths uses a certain sub-semi-ring of $\End(\mathbb{T})$, which we now define.

\begin{definition}[Non-Decreasing Endomorphisms]\label{def:time-increasing-semi-ring}
    Let $\mathbb{W}$ be the set of all endomorphisms \linebreak ${w: \mathbb T \to \mathbb T}$ such that $w$ is non-decreasing and $\lim\limits_{t\to \infty}{w(t)} = \infty$. We define the operation $\oplus$ and $\otimes$ on $\mathbb W$ as follows: for all $t \in \mathbb T$,
\begin{itemize}
    \item $(w\oplus v )(t) = \min\{w(t), v(t)\} $, and
    \item $(w\otimes v )(t) = w(v(t))$.
\end{itemize}
\end{definition}

\begin{remark}
    The intuition behind \Cref{def:time-increasing-semi-ring} is that a function $w_{ij}\in \mathbb{W}$ dictates when a message transmitted at time $t$ from node $i$ will arrive at node $j$---this is the time $w_{ij}(t)$.
    Considering the Kleene star of such a matrix will also solve the earliest time of arrival for a traffic network, where the density of traffic can vary over time.
    The non-decreasing condition in \Cref{def:time-increasing-semi-ring} is also clear in the traffic example, as departing later in the day will never result in an earlier time of arrival.
\end{remark}

\begin{remark}[Connections to Delay Tolerant Networking]\label{rmk:tropical-DTN}
    \cite[pg.33]{Baras2010} observes that \Cref{def:time-increasing-semi-ring} can also be used to model opportunistic networks and delay tolerant networking, which is the current paradigm for space networking.
    Indeed, one can embed our lifetime semi-ring into this model as well, but will sacrifice the ease of use and the connections to topological data analysis described in \Cref{sec:TDA-and-metrics}.
\end{remark}

Finally, we remark that capacity constraints can be added to \Cref{def:time-increasing-semi-ring} without issue.

\begin{remark}[Capacity Constrained Time-Variate Routing]
    Using a similar approach to what was described above, we can construct an endomorphism-based semi-ring that captures capacity-constrained delivery on time-varying networks. 
    Specifically, given a capacity $C > 0$, for each $w \in \mathbb W$, there exists an $W \in \mathbb W$, such that:
    \[W(t) = \min\big\{s : \int\limits_{t}^{s}{w(x)dx}= C \big\}.\]
    One can then compose these elements just as in \Cref{def:time-increasing-semi-ring}.
\end{remark}

%\tung{those are just semi-rings. will write later on the exposition.}
%transfering packet with certain size in space network communication setting where the link depends on availibility and the rate/capacity
%integrations of throughput rate over available time of a link.

\subsection{The Universal Contact Semi-Ring (UCS) for Time-Variate Routing}\label{sec:UCS}

We finish our discussion of semi-rings by introducing a novel semi-ring 
%Contact Multigraph algorithm, we introduce a generalized semi-ring, 
called the \define{Universal Contact Semi-Ring (UCS)} that allows us to model a large class of networking problems.
Most significant for our application of semi-rings to space networking is the observation that the UCS allows us to model the "store and forward" protocol \cite{rfc4838}, which is the main paradigm for delay tolerant networking (DTN) and cannot be captured via the other semi-rings described above.
After providing the necessary definitions, we work through a synthetic example of how this semi-ring describes the store and forward networking protocol.

\begin{definition}\label{def:UCS}
The \define{Universal Contact Semi-Ring (UCS)} $\mathcal{C}$ consists of a collection of maps from $\R$ to the powerset of $\,\R$,
\[
    \mathcal{C} = \{f:\R \to 2^{\R} \},
\]
together with addition and multiplication operations defined as
%Our Contact Multigraph semi-ring $\mathcal C$ is composed of functions from $\R$ to the powerset of $\R$: $\mathcal C = \{\R \to 2^\R\}$, and has the following addition and multiplication laws:
\[f \oplus g(t) = f(t) \cup g(t) \quad \text{ and } \quad f \odot g (t) = \bigcup_{x \in f(t)}g(x + t) + x, \]
%where the additions $g(y) + y$ being in the sense of Minkowski, i.e.
respectively. Here the formal sum $A + x$ denotes the Minkowski sum, i.e
\[ A + x = \{a + x, a \in A\}.\]
The additive identity of $\mathcal C$ is the constant function 
\begin{align*}
    0_{\mathcal C}: \R & \to 2^{\R} \\
                    t & \mapsto \emptyset,
\end{align*}
while the multiplicative identity is the constant function
\begin{align*}
    1_{\mathcal C}: \R & \to 2^{\R} \\
                    t & \mapsto \{0\}.
\end{align*}
\end{definition}

%Under this structure the set $\mathcal C$ is an idempotent semi-ring. 
%This is a non-commutative, non-integral, and non-totally ordered semi-ring, and is a good example of a poorly behaved semi-ring with real world applications.

Although the Universal Contact semi-ring is idempotent, this semi-ring is poorly-behaved due to its lack of commutative and integral properties, as well as a total order.
However, we believe the UCS is suitable for modeling in real world application:
In the context of networking, a map $f \in \mathcal C$ represents the possible delivery delays at a given time: if $x \in f(t)$, then a message can be transported along the edge to arrive at time $t + x$. 
This allows the UCS to encapsulate contact windows, time forwarding, changing communication times, as well as storage---limited or unlimited---at a node.

%This is a very generalized semi-ring, and contains many of the other semi-rings which we have discussed as substructures -- most of which can be expressed as quotients of sub-semi-rings of $\mathcal C$.
Moreover, UCS contains most of the semi-rings that naturally arise in routing problems, described above, as subquotients, i.e., quotients of sub-semi-rings of UCS. The proof of the next result is presented under \Cref{prop:UCS-sub-quot-appendix} in the appendix.

\begin{proposition}\label{prop:UCS-sub-quot}
Let $\mathcal C$ denote the Universal Contact Semi-Ring (\Cref{def:UCS}). $\mathcal C$ contains 
    \begin{enumerate}
        \renewcommand{\theenumi}{(\alph{enumi})}
        \item an injective image of the boolean semi-ring $S$; % Do we call the boolean semi-ring $S$ elsewhere?
        \item a sub-semi-ring which is isomorphic to the TVG semi-ring;
        \item a sub-semi-ring which surjects onto the tropical semi-ring $\mathbb T$;
        \item a sub-semi-ring which surjects onto the propagation delay semi-ring ; and
        \item a sub-semi-ring which surjects onto the function endomorphism semi-ring $\mathbb W$.
    \end{enumerate}
\end{proposition}

\begin{example}\label{ex:store-and-forward} 
\begin{figure}[h!]
\centering
\begin{tikzpicture}[->,auto,node distance=3cm,
      thick,main node/.style={circle,draw,font=\sffamily\Large\bfseries},scale=1.5]
      %Author: Jacob Cleveland
    
      \node[main node] (1) at (-2,0) {$A$};
      \node[main node] (2) at (0,1) {$B$};
      \node[main node] (3) at (0,-1) {$C$};
      \node[main node] (4) at (2,0) {$D$};
    
      \path[every node/.style={font=\sffamily\small}]
          (1) edge[black] node [above=3mm] {([0,1]),10)} (2)
          (1) edge[black] node [below=3mm] {([0,1],0)} (3)
          (2) edge[black] node [above=3mm] {([10,11],5)} (4)
          (3) edge[black] node [below=3mm] {([5,6],0)} (4);
    \end{tikzpicture}
    \caption{Time-varying network of 4 nodes: each edge represents a connection,\\ decorated with available time and delay time.}
    \label{fig:store}
\end{figure}
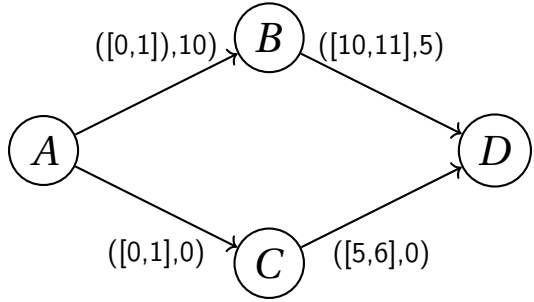
Let $\mathcal N$ be a time-varying network with 4 nodes $A, B, C, $ and $D$ as described in \Cref{fig:store}.
One can endow this network with the universal contact semi-ring $\mathcal C$ by treating the edge $([a,b], \omega)$ as the map:
\[ t \mapsto \begin{cases}\{\omega\} & \text{for  } t \in [a,b]\\ \emptyset & \text{ otherwise.}\end{cases} \]
In order to model the potential storage of information at a node, we add a self-loop on each node, representing via the maps that sends $t \mapsto [0, \infty)$.
%Adding self-loops to each node with the map that sends $ t \mapsto [0, \infty) $ allows us to model the potential storage of information at a node. 
Using the propagation delay semi-ring, or our other semi-rings, the only possible route from $A$ to $B$ would be the route $A\to B\to D$ with delay $15$.
However, in $\mathcal C$, the route $A \to C \to C \to D$ is permitted and has delay $5$.

The semi-ring weight of the route $A \to B \to D$, obtained by multiplying the edge weights in the path, is given by the map:
\[ t \mapsto \begin{cases}\{15\} & t \in [0,1]\\ \emptyset &t \not\in [0,1]\end{cases}\]
while the semi-ring weight of the route $A \to C \to C \to D$ will be given by the map that sends
\[t \mapsto \begin{cases}[5-t,6-t] & t \in [0,1]\\ \emptyset & t \not\in [0,1]\\\end{cases}\]
So in addition to allowing for another route, assuming that nodes are given a storage buffer, we also can see that in terms of total route time, for the route $A \to C \to C \to D$, we can send our message at any time within $[0,1]$ and still arrive by time $5$.

We note that limited storage buffer at each node can be encapsulated by modifying the self-loop maps as $t \mapsto [0, a]$, where  $[0, a]$ is a finite interval.
%modify $ t \mapsto [0, \infty) $ to encapsulate the node buffers
\end{example}

\begin{figure}[h]
    \centering
    \includegraphics[width=\textwidth]{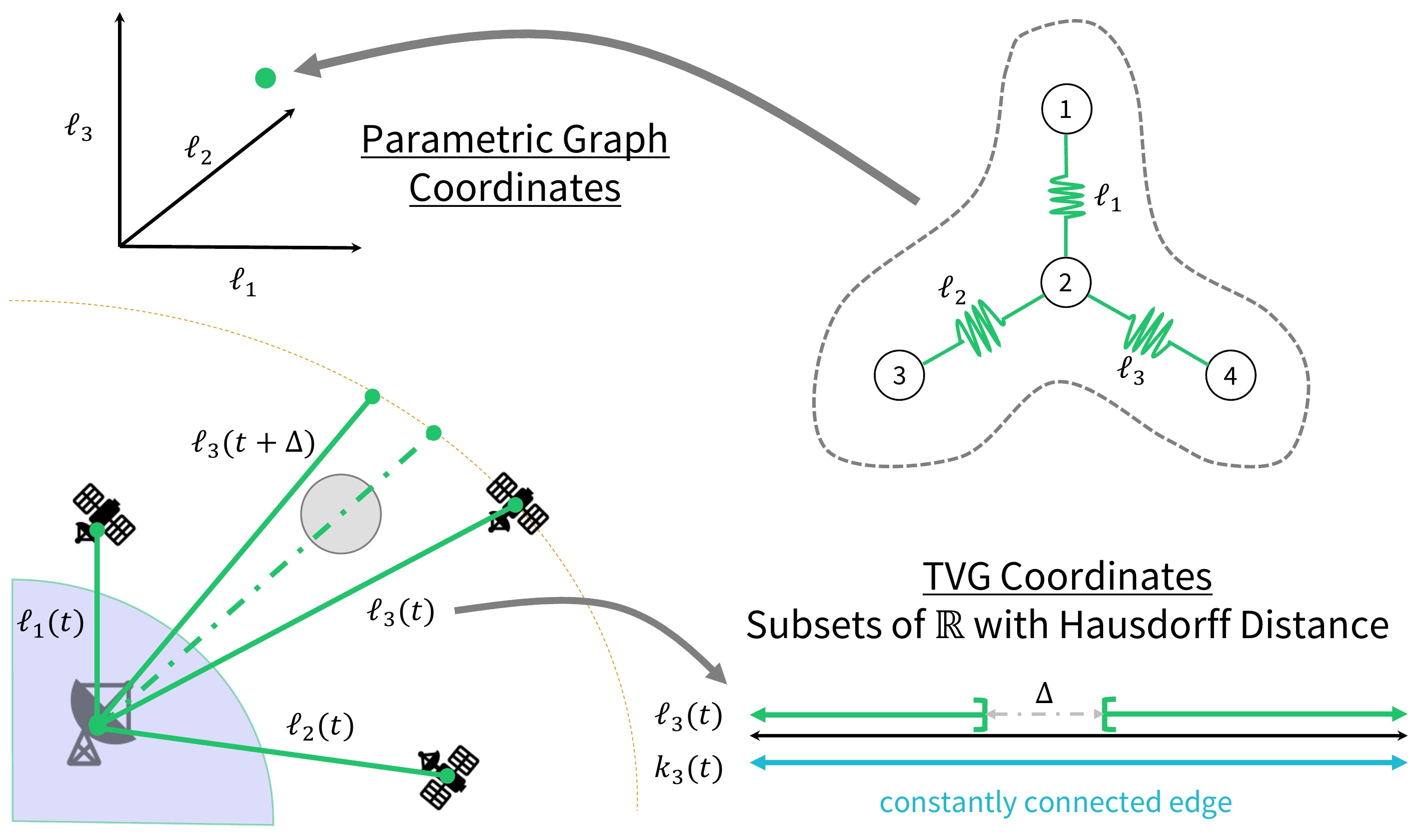}
    \caption{Models that treat edge lengths or weights as continuously varying parameters are fundamentally ill-equipped to handle removal of edges or other changes in graph topology. Indeed, if $\ell_3$ were removed in the parametric model, the value of the third coordinate becomes ambiguous. By contrast, the TVG model coordinatizes disruptions in service appropriately and the Hausdorff distance can be used to measure duration of disconnect.}
    \label{fig:param-vs-TVG-coordinates}
\end{figure}

\section{Geometric and Topological Models for Time-Varying Graphs}\label{sec:TDA-and-metrics}

In this section we take up the question of when two time-varying graphs (TVGs) are ``close'' in a precise sense.
This is an important question for a variety of reasons, but one immediate application of such a notion would be a calculation of whether an $n$-node STARLINK sub-system is actually close to a strongly connected system, as the lifetime curves in \Cref{fig:lifetime-curves-sample} suggest.
Lifetime curves are merely a summary statistic and lack the discriminatory power of a true metric, but \Cref{fig:STARLINK-sample-distances} supports the intuition that (probabilistically) more than 40 nodes is sufficient for constructing a sub-TVG that is strongly connected, with $n=100$ providing a higher probability for a guarantee on constant connectivity.

More broadly, the TVG model fits nicely within a growing topological data analysis (TDA) paradigm, where topological changes can be quantified precisely.
This is a fundamental improvement over traditional parametric graph models such as those used in \cite{joswig2022parametric}, where the removal of an edge causes a discontinuity in the coordinatization process; see \Cref{fig:param-vs-TVG-coordinates}.

% \todo{Justin says: Make 3.1 about Strict Hausdorff distance between TVGs. This is computable thanks to the theorem of Woojin, et al that says $d_H=d_B(complement windows)$ and demo our distance to the strongly connected TVG as a function of walk length. THEN, a SHORT Section 3.2 on symmetrized Hausdorff distance with a few stated results, such as the fact that interleaving distance is equal to the symmetrized Hausdorff, but supress proof and even statement of interleaving to the Appendix. FINALLY Section 3.3 will be on TDA summaries, since Section 3.2 stuff is uncomputable: KNN classification and all that good stuff.}

% %\section{Summary Graphs and Stability Theorems (Cosheaf Models)}
% \todo{Tung: name change suggestion for this section, it should focus on summary graphs and barcodes. Hausdorff-like distances and stability should be in a seperate section. Thoughts?}

% GOAL OF THIS SECTION IS TO PROVE THAT THE TVG TO ZZ BARCODE PIPELINE IS STABLE. ALSO DEFINE TWO METRICS ON TVGS AND RELATE THEM TO A "BOTTLENECK TYPE" DISTANCE ON TVGS USING THE COMPLEMENTS OF TEH CONTACT WINDOWS, VIEWED AS PERSISTENCE DIAGRAMS.

% \todo{define Hausdorff version of entry-wise on TVGs, defined the sym version and prove the zigzag stability in appendix.}

% \todo{Figure for illustrating summary graph and cosheaf construction. BOB MAYBE HAS ONE ALREADY?}

\subsection{Distances on Time-Varying Graphs}\label{sec:distances-on-TVGs}
%\todo{For strict Hausdorff distance on TVGs, I claim that the maximum over entry-wise bottleneck distances for the complement of connection windows is equal to the strict Hausdorff distance.}
In this section we begin the study of distances on TVGs (\Cref{defn:TVG}). 
We use distances rather than metrics because certain TVGs may be infinitely far away and also two TVGs may have distance zero, even if they are not exactly the same.
We briefly recall this more flexible notion of a metric. 
\begin{definition}[Distance]
 A \define{distance} on a set $X$ is a map $d: X \times X \to [0, \infty]$ where for all $x, y \in X$: 
\begin{enumerate}
    \item $d(x, y) = d(y, x)$
    \item $d(x, y) \geq 0$ and $d(x, x) = 0$
\end{enumerate}   
\end{definition}

%We do this in order to bound distances between certain topological summaries of TVGs, such as Reeb graphs and barcodes, in terms of natural distances that are defined on TVGs directly.
%A crucial idea for studying stability of the TVG is the summary graph construction.

% We provide a notion of distance between two TVGs via comparing their matrix TVGs.
% Recall that for a simple time-varying graph $\cG=(G,\ell_M)$ of $n$ nodes,  $M(\cG)$ is a matrix of size $n\times n$ associated to $\cG$, whose entries are subsets of $\R$. 

Fundamental to our distances on TVGs is the Hausdorff distance.
%We note that one can measure the the distance between a pair of subsets of $\R$ via Hausdorff distance: 
\begin{definition}\label{def:Hausdorff-distance}
   Given $A, B \subseteq \R$, the \define{Hausdorff distance $d_H$} between $A$ and $B$ is defined as 
\[
    d_H(A, B) =  \inf\{\varepsilon \geq 0 |\; A \subseteq B^\varepsilon \text{ and } B \subseteq A^\varepsilon\}, 
\] where $A^\varepsilon = \bigcup\limits_{a\in A}{\text{Ball}_\varepsilon(a)}$ and $B^\varepsilon = \bigcup\limits_{b\in B}{\text{Ball}_\varepsilon(b)}$. 
Here ${\text{Ball}_\varepsilon(a)}=\{x \in \R\mid |x-a|\leq \varepsilon\}$.
\end{definition}

\subsubsection{Distances with Fixed Node Correspondence}\label{sec:distances-fixed-node}

\begin{definition}\label{def:Hausdorff-TVG-distance}
Given two matrix TVGs $M,N\in \Mat_n(\powR)$ their \define{Hausdorff distance} is
     \[
     d_H(M,N) := \max\limits_{1 \leq i, j \leq n}{\left\{ d_H \left(M_{ij}, N_{ij}\right)\right\}}    .\] 
\end{definition}

\begin{figure}[h]
    \centering
    \includegraphics[width=\textwidth]{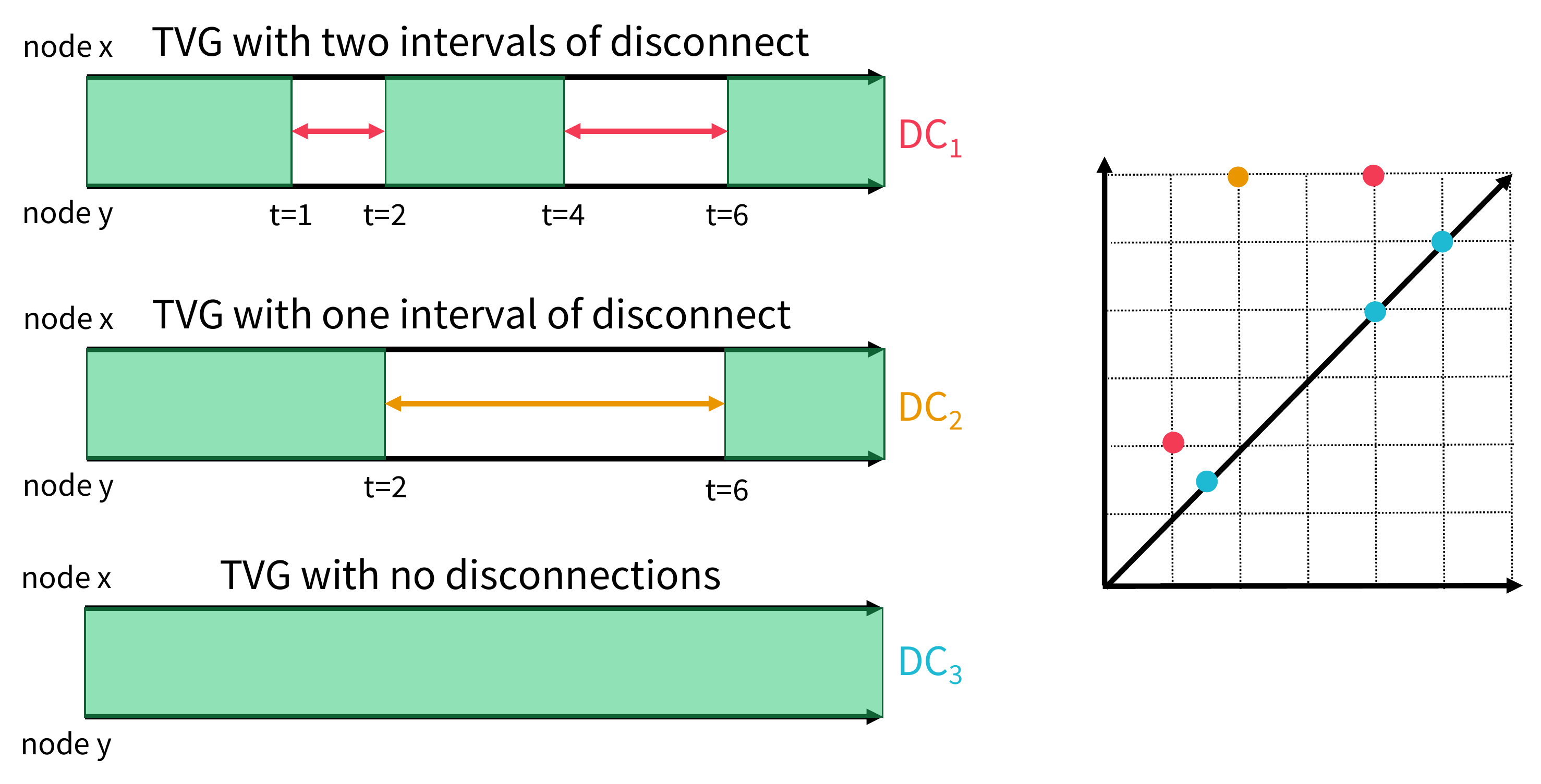}
    \caption{From an algorithmic perspective it is better to view the Hausdorff distance on windows of connection time as the Bottleneck distance on the disconnect times. This leverages an isometry proved in \cite[Thm.~E.1]{kim2019interleaving}. A TVG with no disconnections may be viewed as having an arbitrary number of points on the diagonal. The Bottleneck distance optimizes over wasy of aligning the disconnects between two TVGs.}
    \label{fig:disconnect-distance}
\end{figure}

%\robert{Add remark about the triangle inequality holding.}
\begin{remark}[Computability of $d_H(M, N)$]\label{rmk:Hausdorff-computability} Assuming $M$ ane $N$ are lifetime matrices (\Cref{defn:matrix-TVG}) then each entry is a finite collection of disjoint closed intervals. 
In this case, the computability of $d_H(M_{ij}, N_{ij})$ reduces to the the computability of the Hausdorff distance between the collections of these intervals. 
\cite[Thm.~E.1]{kim2019interleaving} proves that this quantity can be obtained via computing bottleneck distance between their respective complement intervals, i.e.,
\[d_H(M_{ij}, N_{ij}) = d_B (C(M_{ij}), C(N_{ij})),\]
where $C(M_{ij})$ and $C(N_{ij})$ denote the complements in $\R$---usually intersected with some compact interval indicating the simulation time---respectively.
\end{remark}

In view of \Cref{rmk:Hausdorff-computability}, we introduce a new suite of distances on TVGs.

\begin{definition}[$(p,q)$-Disconnect Distances]\label{def:p-q-disconnect-distance}
    Given lifetime matrices $M,N\in \Mat_n(\subR)$ their \define{$(p,q)$-Disconnect Distance} is
     \[
     d_{DC,p,q}(M,N) := \left\Vert \left\langle d^p_{W} \left (C(M_{ij}), C(N_{ij})\right )\right\rangle_{1 \leq i, j \leq n}\right\Vert_q    .\] 
     Where $C(M_{ij})$ and $C(N_{ij})$ is the complement of the intervals in $M_{ij}, N_{ij} \in\subR$, respectively, and $d^p_W$ is the Wasserstein distance on barcodes, reviewed in the Appendix. These distances populate the entries of a length $n^2$ vector, whose $\ell^q$ norm is then computed.
\end{definition}

\begin{remark}
    Following \Cref{rmk:Hausdorff-computability}, when \Cref{def:Hausdorff-TVG-distance} is restricted to lifetime matrices the Hausdorff distance can be viewed as a special case of \Cref{def:p-q-disconnect-distance} when $p=q=\infty$.
\end{remark}

In \Cref{fig:STARLINK-sample-distances} we compute the disconnect distances for $p=q=\infty$ (Hausdorff distance) and $p=q=2$ (the 2-Wasserstein distance).

\begin{figure}[h]
\begin{center}
\begin{tikzpicture}[baseline]
\begin{axis}[
    width=5.25cm,
    height=5cm,
    title={Bottleneck Distance \\
        {\footnotesize $k$-Walk to Complete}},
    xlabel={$k$},
    ylabel={seconds},
    xmin=1, 
    ymin=0,
    % ymax=1.2,
    % legend pos=south east
    % y tick label style={/pgf/number format/.cd,fixed,precision=5,/tikz/.cd},
    % scaled y ticks = false,
    % axis lines=left,
    % cycle list/Paired
]
% \foreach \yindex in {1,...,3} {
%     \addplot+[ultra thick] table[y index = \yindex]
%         {simulations/knn_single_overlay_today.dat};
% }
\foreach \yindex/\nval in {2/30,3/40,4/50,5/70,6/100} {
    \edef\temp{\noexpand\addlegendentry{$n=\nval$}}
    % \addlegendentry{$n=\nval$};
    \addplot+[ultra thick] table[y index = \yindex]
        {simulations/bottleneck_distances_100_sats_today.dat};
    \temp
}
\end{axis}
\end{tikzpicture}
\qquad \qquad
\begin{tikzpicture}[baseline]
\begin{axis}[
    width=5cm,
    height=5cm,
    title={$2$-Wasserstein Distance \\
        {\footnotesize $k$-Walk to Complete}},
    xlabel={$k$},
    ylabel={seconds},
    xmin=1, 
    ymin=0,
    % ymax=1.2,
    % ytick=\empty,
    % legend pos=south east
    % y tick label style={/pgf/number format/.cd,fixed,precision=5,/tikz/.cd},
    % scaled y ticks = false,
    % axis lines=left,
    % cycle list/Paired
]
% \foreach \yindex in {1,...,3} {
%     \addplot+[ultra thick] table[y index = \yindex]
%         {simulations/knn_single_overlay_average.dat};
% }
\foreach \yindex/\nval in {2/30,3/40,4/50,5/70,6/100} {
    \edef\temp{\noexpand\addlegendentry{$n=\nval$}}
    % \addlegendentry{$n=\nval$};
    \addplot+[ultra thick] table[y index = \yindex]
        {simulations/wasserstein_distances_100_sats_today.dat};
    \temp
}
\end{axis}
\end{tikzpicture}
\end{center}
\caption{Two different metrics help verify the summary statistics depicted in \Cref{fig:lifetime-curves-sample}---for each sample of $n=30,40,50,70,100$ nodes from STARLINK, the distance from the $k$-cumulant to the constant matrix with value 86400 is computed for increasing $k$. The Bottleneck Distance (left) measures the largest deviation from constant connectivity, and separates out $n=100$ as the one truly connected subnetworks, whereas the 2-Wasserstein distance sums differences across entries in the TVGs and more tightly clusters the $n=30$ and $40$ node simulations apart from the $50$ and higher number node systems.}
    \label{fig:STARLINK-sample-distances}
\end{figure}
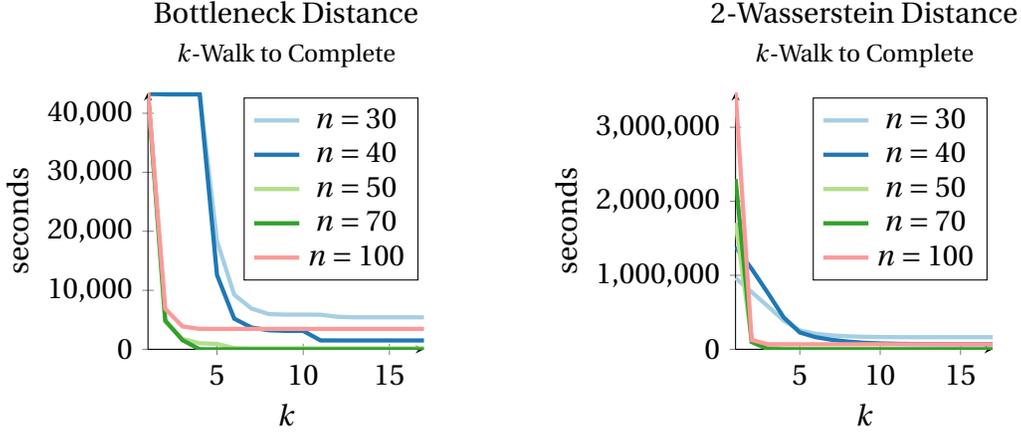

% \begin{figure}[h]
%     \centering
%     \includegraphics[width=0.49\textwidth]{aggregate-bottleneck.png}
%     \includegraphics[width=0.49\textwidth]{aggregate-2-wasserstein.png}
%     \caption{Two different metrics help verify the summary statistics depicted in \Cref{fig:lifetime-curves-sample}---for each sample of $n=20,30,40,50,70,100$ nodes from STARLINK, the distance from the $k$-cumulant to the constant matrix with value 86400 is computed for increasing $k$. The Bottleneck Distance (left) measures the largest deviation from constant connectivity, and separates out $n=100$ as the one truly connected subnetworks, whereas the 2-Wasserstein distance sums differences across entries in the TVGs and more tightly clusters the $n=20$ and $30$ node simulations apart from the $40$ and higher number node systems.}
%     \label{fig:STARLINK-sample-distances}
% \end{figure}

\subsubsection{Distances with Unknown Node Correspondence}\label{sec:distances-w-unknown-node}

Both Definitions \Cref{def:Hausdorff-TVG-distance} and \Cref{def:p-q-disconnect-distance} are only useful when the number of nodes in two TVGs are the same and there is a fixed node correspondence. This is because the vertices need to be ordered to determine a matrix representation and then distances are computed entry-by-entry for these.
In order to consider the distance between two general TVGs---where node correspondence is unclear---we need to introduce a relaxed Hausdorff distance, which has provable connections with an interleaving-type distance that is popular in TDA.

\begin{definition}\label{def:symmetric-hausdorff-distance}
The \define{symmetrized Hausdorff distance} between two matrix TVGs $M,N\in \Mat_n(\powR)$ is
    \[
        d_{\Sigma H}{(M, N)} = \inf\limits_{\sigma \in \Pi(n)}{d_H(M, \sigma(N))}
    ,\] where $\Pi(n)$ is the set of all permutations of $\{1,\ldots, n\}$, and $\sigma(N)$ is the matrix TVG whose $(i,j)$-th entry is $N_{\sigma(i),\sigma(j)}$. 
\end{definition}

% The following is a simple corollary:

\begin{corollary} 
For any pair of matrix TVGs $M,N\in \Mat_n(\powR)$, one has 
    \[
        d_{\Sigma H}{(M, N)} \leq d_H(M, N) 
    .\] 
\end{corollary}

\begin{remark}
    Since $d_{\Sigma H}$ requires considering $d_H$ over all $n!$ possible permutations, $d_{\Sigma H}$ tends to be prohibitively expensive to compute, which is why we will work with topological summaries of TVGs such as zigzag barcodes instead.
\end{remark}

\subsubsection{Connections with the Interleaving Distance via Cosheaves}\label{sec:connect-w-interleaving-cosheaves}

Similar to the snapshot construction of \Cref{def:snapshot}, one can summarize a TVG over any interval $I\subseteq \R$.
Moreover, to better understand how a TVG behaves over a given interval $I$ and the relationship of this summary to another summary $J$, where $I\subseteq J$, we introduce a functorial summary of a TVG that has the added benefit of being a cosheaf.
The finer details of cosheaf theory are not needed for this paper and a basic understanding of category theory, as reviewed in \Cref{sec:basic-category-theory}, is sufficient. 
However, further reading on category theory is always encouraged and \cite{cats-in-context} in particular is an excellent resource.

\begin{definition}[Summary Graphs]\label{def:summary_graph}
Let $\cG=(G,\ell_M)$ be a time-varying graph in the sense of \Cref{defn:TVG}.
To each interval $I\subseteq \R$ we have a \define{summary subgraph over $I$}:
\[
\mathcal{E}_M(I):=\{e\in G \mid \ell_M(e)\cap I \neq \varnothing\}.
\]
This is the subgraph of $G$ where an edge is included iff it is alive at some point in the interval $I$.
The \define{underlying graph} of a time-varying graph $\cG$
is the summary graph over $I=\R$.
\end{definition}

We now promote the summary graph construction to a functor with two possible codomains.

\begin{definition}[The Category of Subgraphs and Graph Monomorphisms]\label{def:subgraph-categories}
    Fix $G=(V,E)$ a simple directed graph.
    \begin{itemize}
        \item The collection of \define{subgraphs of $G$}, written $\sub(G)$, has for objects pairs $G'=(V',E')$ where $V'\subseteq V$ and $E'\subseteq E$. There is a unique morphism from $G'\to G''$ iff $V'\subseteq V''$ and $E'\subseteq E''$, thus making $\sub(G)$ into a poset.
        \item The collection of \define{graph monomorphisms to $G$}, written $\mon(G)$, consists of simple directed graphs $G'=(V',E')$ along with an injective graph morphism $\varphi':G'\to G$. There is a morphism $\psi:(G',\varphi')\to (G'',\varphi'')$ if $\psi$ is a graph morphism satisfying $\varphi''\circ \psi = \varphi'$.
        \item There is a functor $\boldsymbol{\iota}:\sub(G)\to \mon(G)$ that takes each subgraph to its corresponding inclusion map.
    \end{itemize}
\end{definition}

\begin{remark}
    The difference between $\sub(G)$ and $\mon(G)$ is subtle, but crucial.
    To see the difference, suppose $G$ is a directed 3-cycle with vertex set $V=\{x,y,z\}$ and edge set \linebreak $E=\{[xy], [yz],[zx]\}$.
    In $\sub(G)$ there is only one object with 3 vertices and 3 edges, namely, $G$ itself.
    However, in $\mon(G)$ there are three objects with 3 vertices and 3 edges, where the injective graph morphisms range over all three cyclic permutations of the 3-cycle $G$.
\end{remark}

%\todo{motivate cosheaves again?}

\begin{definition}[Summary Cosheaf]\label{def:summary-cosheaf}
    The \define{summary cosheaf} of a TVG $\cG=(G,\ell_M)$ is the functor
    \[
        \mathcal{E}_M: \Int \to \sub(G) \qquad \text{where} \qquad I \mapsto \mathcal{E}_M(I)\subseteq G,
    \]
    which assigns to every closed interval $I$ the summary graph over $I$.
\end{definition}

\begin{remark}\label{rmk:summary-cosheaf}
Following \cite{jmc,crg}, one can check that $\mathcal{E}_M$ is actually a cosheaf, although that observation is not necessary here.
Additionally, the functor $\mathcal{E}_M$ is one way of viewing that our TVGs are equivalent to the dynamic graphs construction of \cite{kim2017extracting}.
\end{remark}

% \begin{proof}
% Let $I \in \Int$ be an open interval and let $\U_I := \{I_r\}$ be an open cover of 
% $I$ by connected open intervals. We want to show that $G(I)$ satisfies the 
% coequalizer condition,
% $$
% \begin{tikzcd}
% \bigsqcup_{p, q} G(I_{p, q})
% \ar[r, yshift=.7ex, "f"] \ar[r, yshift=-.7ex, "g"'] &
% \bigsqcup_r G(I_r) \ar[r, "\eta"] \ar[rd, "\varepsilon"'] 
% & G(I) \ar[d, dashed, "\exists! \underline \varepsilon"] \\ 
% & & X
% \end{tikzcd}
% $$
% given $\varepsilon : \sqcup_r G(I_r) \to X$ with $\varepsilon f = \varepsilon g$, 
% there exists $\underline \varepsilon : G(I) \to X$ such that 
% $\underline \varepsilon \circ \eta = \varepsilon$.

% If $e \in G(I)$ then it has weight $< \delta$ for some open interval $J$ in $I$. 
% This interval $J$ must intersect some $I_r$ nontrivially. Define 
% $\underline \varepsilon(e) := \varepsilon_r(e)$. We must show that this is 
% well-defined : if $e$ has weight $< \delta$ during intervals $I_p$ and $I_q$,
% then since $\varepsilon f = \varepsilon g$, we have that 
% $\varepsilon f_{pq}(e) = \varepsilon g_{pq}(e)$, which is equivalent to 
% $\varepsilon_p(e) = \varepsilon_q(e)$, as desired.
% To see that $\underline \varepsilon$ is unique, suppose some $\gamma : G(I) \to X$
% has the property that $\gamma \eta = \varepsilon$. Then for $e \in G(I_p)$, we 
% have $\gamma \eta(e) = \gamma(e) = \varepsilon_p(e)$, establishing that 
% $\underline \varepsilon = \gamma$.
% \end{proof}

% Given TVG's $\cG = (G, \ell_M)$ and $\cG' = (G, \ell_M')$ over the same 
% underlying graph $G$, there is a natural interleaving distance between them.

\begin{definition}
If $\cG=(G,\ell_M)$ and $\cG'=(G,\ell_N)$ are two TVGs with the same underlying graph, 
then they are 
\define{$\varepsilon$-interleaved}
if there exist morphisms 
$\varphi : \mathcal{E}_M \to \mathcal{E}_N$ and 
$\psi : \mathcal{E}_N \to \mathcal{E}_M$ 
such that for all $I \in \Int$, the following 
diagram commutes:
$$
\begin{tikzcd}[row sep=normal, column sep=huge]
\mathcal{E}_M(I) 
\ar[r]  \ar[rd, "\varphi_I" near start] & 
\mathcal{E}_M(I^\varepsilon) 
\ar[r] \ar[rd, "\varphi_{I^\varepsilon}" near start] & 
\mathcal{E}_M(I^{2 \varepsilon}) \\
\mathcal{E}_N(I) 
\ar[r] \ar[ru, crossing over, "\psi_I"' near start] & 
\mathcal{E}_N(I^\varepsilon)
\ar[r] \ar[ru, crossing over, "\psi_{I^\varepsilon}"' near start] & 
\mathcal{E}_N(I^{2 \varepsilon}),
\end{tikzcd}
$$
where $I^{\varepsilon}$ denotes the $\varepsilon$-thickening of $I$, cf. \Cref{def:Hausdorff-distance}.
The \define{interleaving distance} between % $\mathcal{E}_M$ and $\mathcal{E}_N$ is
$\cG$ and $\cG'$ is
\[d_I(\cG, \cG') = 
\inf \big \{ 
\varepsilon \geq 0 : \mathcal{E}_M \text{ and } \mathcal{E}_N 
\text{ are } \varepsilon \text{-interleaved} 
\big \}.\]
\end{definition}

 \begin{theorem}[Isometry Theorems]\label{thm:isometries}
 Suppose $\cG=(G,\ell_M)$ and $\cH=(H,\ell_N)$ are two TVGs with isomorphic underlying graphs.
 Fix an isomorphism $\Phi: H \to G$, which in turn induces isomorphisms $\sub(G)\cong \sub(H)$ and $\mon(G)\cong \mon(H)$, so we can view $\mathcal{E}_N$ and $\boldsymbol{\iota}\circ \mathcal{E}_N$ as functors to $\sub(G)$ and $\mon(G)$, respectively.
 Any ordering of the nodes in $G$ then determines two isometries
     \[
     d_{H} (M, N) = d_{I} (\mathcal{E}_M, \mathcal{E}_N) \quad \text{and} \quad d_{\Sigma H} (M, N) = d_{I} (\boldsymbol{\iota}\circ\mathcal{E}_M, \boldsymbol{\iota}\circ\mathcal{E}_N),
     \] 
 where $M$ and $N$ are the lifetime matrices for $\cG=(G,\ell_M)$ and $\cH=(H,\ell_N)$, respectively.

%    \begin{proof}
%    \begin{description}
%        \item[($\leq$)] Suppose $\cG, \cH$ are $\varepsilon$ interleaved, let $(\varphi, \psi)$ form an $\varepsilon$-interleaving pair between 
%$\cG$ and $\cH$. In particular, they represent an isomorphism between two underlying 
%graphs with the same number of nodes.
% 
%        \item[($\geq$)]
%        
%    \end{description}
%    \end{proof}
\end{theorem}

%\todo{Robert + Tung, please write a proof of the above statements.}

\subsection{Topological Summaries of TVGs via Barcodes for Machine Learning}\label{sec:top-summaries-TVGs}

% In this section, we introduce the topological summaries of TVGs using the notions of barcodes obtained from persistence homology. 
% Heuristically, for a graph $G$, one can capture captures the homological information at the degree $k$ of a subgraph of $G$ via the homological functor $H_k$:

In this section we show how techniques from topological data analysis (TDA) can be used to simplify the study of TVGs, by converting them into two zigzag persistence barcodes: one for degree-0, which summarizes how components in the snapshots of $\cG$ evolve over time, and one for degree-1, which summarizes how cycles in the snapshots evolve.
Two important conclusions of this section are
\begin{enumerate}
    \item The degree-0 and degree-1 zigzag barcodes are \emph{stable features} of TVGs, and
    \item These features can be used to distinguish Earth-Mars vs. Earth-Moon space networking scenarios in a K-Nearest Neighbors (KNN) classifier, \emph{using the network topology alone.}
\end{enumerate}

These zigzag barcodes are based on the homology of a graph, along with maps that are induced on homology from graph morphisms.

\begin{definition}[Homology]
   Fix a field $\Bbbk$ for the remainder of the paper.
    To every graph $G$, one has two vector spaces associated: $H_0(G)$---the vector space generated by the connected components of $G$---and $H_1(G)$---the vector space generated by cycles in $G$. 
    Additionally, associated to any graph morphism $G'\to G$ there are induced linear maps $H_k(G')\to H_k(G)$ on these vector spaces.
    In otherwords, we have the \define{homology functors} for $k=0$ and $k=1$:
    \[
    H_k:\mon(G) \to \vect.
    \]
\end{definition}

\begin{definition}[Homology Modules for TVGs]
Recall the subgraph categories of \Cref{def:subgraph-categories} and the summary cosheaf construction of \Cref{def:summary-cosheaf}.
For any TVG $\cG$, we have, 
by post-composing with $\mathcal{E}_M (\cG)$, the associated \define{homology modules} $H_k\mathcal{E}_M (\cG): \Int \to \vect$, which captures the homology of each summary graph of TVG $\cG$ in degrees $k=0$ and $1$.
\end{definition}

\begin{definition}[TVG Barcodes]
    For $k=0,1$, the homology modules $H_k\mathcal{E}_M (\cG): \Int \to \vect$ have a canonically associated \define{zigzag barcode}, which is a multiset of intervals in $\R$:
    \[
        B_k(\cG) = \{(I_j,m_j)\} \quad \text{where} \quad m_j:I_j \to \mathbb{N}.
    \]
    
\end{definition}

\begin{remark}[Zigzag Persistence Review]\label{rmk:zigzag-explanation}
    Zigzag persistence is well studied in the TDA community with computational foundations established in \cite{Carlsson:zigzag:2009, GraphPersistence:2023,GraphZigzag:2021,fastzigzag, dionysus2} among others.
    We review two theoretical perspectives on this construction.
    
    The first proceeds directly by recognizing that associated to each TVG $\cG=(G,\ell_M)$ is a finite set of closed intervals for each edge $v,e\in G$. 
    Taking the union of the endpoints of these intervals across all $v,e$ specifies a finite set of ``critical values'' $\{\tau_i\}_{i=0}^n$
    where an edge or vertex can appear, from this we can express any TVG as a zigzag diagram of sub-graphs of the complete graph
    \[
    G_{-1} \hookrightarrow G_0 \hookleftarrow G_1 \hookrightarrow G_2 \hookleftarrow \cdots \hookrightarrow G_{2n} \hookleftarrow G_{2n+1},
    \]
    where even-indexed subgraphs $G_{2i}$ indicate the snapshot of $\cG$ at each critical value $\tau_i$ and odd-indexed subgraphs $G_{2i\pm 1}$ correspond to snapshots of $\cG$ at times $\tau_i\pm \epsilon$ for sufficiently small $\epsilon$.
    Taking homology of each of these subgraphs and graph inclusions produces a representation of a type $A_{2n+3}$ alternating quiver, which by Gabriel's theorem \cite{gabriel1972} has a canonical multiset of intervals associated to it.

    Alternatively, one can follow the work of \cite{kim2017extracting}, which shows that our TVGs are equivalently viewed as dynamic graphs.
    Following \cite[Def. 2.17]{kim2017extracting} one sees that restricting $H_k\mathcal{E}_M (\cG)$ to one-point intervals $[t,t]$ produces a map $h_k:\R\to\vect$ that is ``cosheaf-inducing.''
    As shown in \cite{zz-stability-bl,hb-stability} constructible cosheaves are block-decomposable and restricting these blocks to the diagonal $y=x$ produces the zigzag barcode described above.
\end{remark}

As previously mentioned, the zigzag barcode is a stable feature of a TVG.
The next proposition explains what is meant by ``stable.''

\begin{proposition}\label{prop:TVG-barcode-stability}
    Suppose $\cG=(G,\ell_M)$ and $\cH=(H,\ell_N)$ are two TVGs with a fixed isomorphism between the underlying graphs $\Phi: G\cong H$, as in \Cref{thm:isometries}.
    The bottleneck distance on the barcodes of $H_k\mathcal{E}_M (\cG)$ and $H_k\mathcal{E}_N (\cH)$ are bounded above as follows:
    \[d_B (\cG, \cH) \leq d_I\left(\boldsymbol{\iota}\circ \mathcal{E}_M (\cG), \boldsymbol{\iota}\circ \mathcal{E}_N (\cH)\right) \leq d_I\left(\mathcal{E}_M (\cG),  \mathcal{E}_N (\cH)\right). \]
    \begin{proof}
    The notation $H_k\mathcal{E}_M$ actually refers to the composition of functors $H_k\circ \boldsymbol{\iota}\circ \mathcal{E}_M$.
    \cite[Proposition 3.6]{categorified-persistence} proves very generally that compositions of functors define Lipschitz-1 maps between categories that are equipped with interleaving distances.
    Consequently
            \begin{align*}
        & \di^{\vect} \big( H_k \circ \boldsymbol{\iota}\circ \mathcal{E}_M (\cG), H_k \circ \boldsymbol{\iota}\circ \mathcal{E}_N(\cH) \big) \\ 
        & \phantom{=========} \leq{} \di^{\mon(G)} \big( \boldsymbol{\iota}\circ\mathcal{E}_M (\cG),  \boldsymbol{\iota}\circ \mathcal{E}_N(\cH) \big) \\ 
        & \phantom{===============} \leq{} \di^{\sub(G)} \big( \mathcal{E}_M (\cG),  \mathcal{E}_M(\cG') \big).
            \end{align*} 
    By the results from \cite{hb-stability, zz-stability-bl}, one has  
        \[d_B (\cG, \cG') = d_I\left(H_k\mathcal{E}_M (\cG), H_k\mathcal{E}_M (\cG')\right),\]
    which proves the stated result.
    \end{proof}
\end{proposition} 

We now demonstrate the discriminative powers of the zigzag barcode on Earth-Moon and Earth-Mars simulations using a K-Nearest Neighbors (KNN) classifier.

\begin{figure}[h]
\begin{center}
\begin{tikzpicture}[baseline]
\begin{axis}[
    width=6cm,
    height=7cm,
    title={Single $k$-NN ($m = 5$) \\
        {\footnotesize 80/20-Split}},
    xlabel={$k$},
    ylabel={Accuracy},
    xmin=1, 
    ymin=0,
    ymax=1.2,
    legend pos=south east
    % y tick label style={/pgf/number format/.cd,fixed,precision=5,/tikz/.cd},
    % scaled y ticks = false,
    % axis lines=left,
    % cycle list/Paired
]
% \foreach \yindex in {1,...,3} {
%     \addplot+[ultra thick] table[y index = \yindex]
%         {simulations/knn_single_overlay_today.dat};
% }
\foreach \yindex/\nval in {1/5,2/10,3/15,5/50} {
    \edef\temp{\noexpand\addlegendentry{$s=\nval$}}
    % \addlegendentry{$n=\nval$};
    \addplot+[ultra thick] table[y index = \yindex]
        {simulations/knn_single_overlay_today_50.dat};
    \temp
}
\end{axis}
\end{tikzpicture}
\qquad \qquad
\begin{tikzpicture}[baseline]
\begin{axis}[
    width=6cm,
    height=7cm,
    title={Average $k$-NN ($m = 5$) \\
        {\footnotesize Over $(100)$ 80/20-Splits}},
    xlabel={$k$},
    xmin=1, 
    ymin=0,
    ymax=1.2,
    ytick=\empty,
    legend pos=south east
    % y tick label style={/pgf/number format/.cd,fixed,precision=5,/tikz/.cd},
    % scaled y ticks = false,
    % axis lines=left,
    % cycle list/Paired
]
% \foreach \yindex in {1,...,3} {
%     \addplot+[ultra thick] table[y index = \yindex]
%         {simulations/knn_single_overlay_average.dat};
% }
\foreach \yindex/\nval in {1/5,2/10,3/15,5/50} {
    \edef\temp{\noexpand\addlegendentry{$s=\nval$}}
    % \addlegendentry{$n=\nval$};
    \addplot+[ultra thick] table[y index = \yindex]
        {simulations/knn_single_overlay_average_50.dat};
    \temp
}
\end{axis}
\end{tikzpicture}
\end{center}
\caption{Average performance curves for a KNN classifier on $H_1$ zigzag persistence barcodes computed from time-varying graphs (TVGs) representing Earth-Moon or Earth-Mars satellite systems. Each simulation has an even class balance of Earth-Moon vs.~Earth-Mars systems, with $m = 5$ samples from each class, and with simulation time the same fixed day in March 2023. For all simulations the same Moon/Mars satellites are used, but the STARLINK satellites are randomly sampled with samples ranging from $s=5,10,15,50$. Average performance is computed over 100 (80\%-20\%) train-test splits.}
\label{fig:KNN-single-day}
\end{figure}
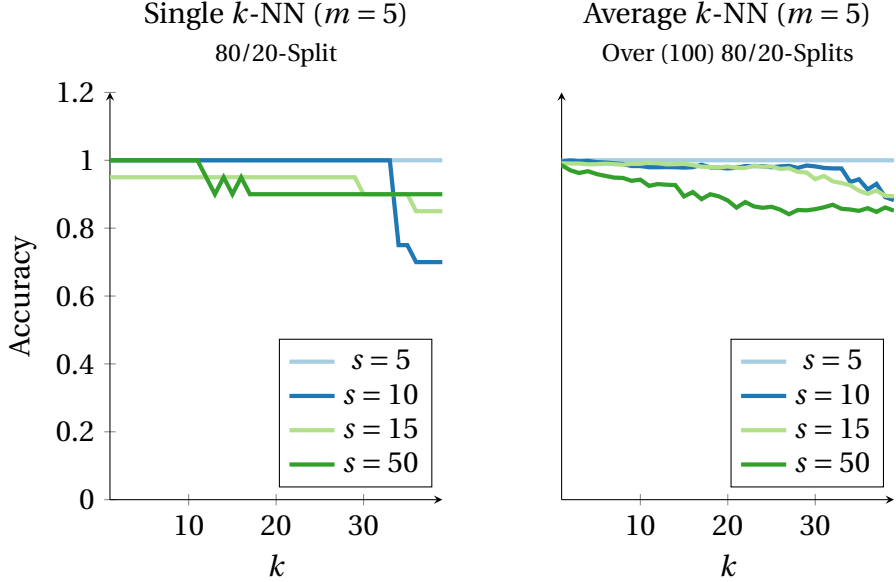

% \begin{figure}[h]
%     \centering
%     \includegraphics[width=0.485\textwidth]{knn-5-starlink-v2-average.png}
%     \includegraphics[width=0.485\textwidth]{knn-10-starlink-v2-average.png}    
%     \includegraphics[width=0.485\textwidth]{knn-15-starlink-v2-average.png}
%     \includegraphics[width=0.485\textwidth]{knn-50-starlink-v2-average.png}
%     % \includegraphics[width=0.55\textwidth]{starlink-confidence-mean}
%     \caption{Average performance curves for a KNN classifier on $H_1$ zigzag persistence barcodes computed from time-varying graphs (TVGs) representing Earth-Moon or Earth-Mars satellite systems. Each simulation has an even class balance of Earth-Moon vs.~Earth-Mars systems, with 50 samples from each class, and with simulation time the same fixed day in March 2023. For all simulations the same Moon/Mars satellites are used, but the STARLINK satellites are randomly sampled with samples ranging from $s=5,10,15,50$. Average performance is computed over 100 (80\%-20\%) train-test splits.}
%     \label{fig:KNN-single-day}
% \end{figure}

\subsubsection{KNN on Earth-Moon vs. Earth-Mars Satellite Systems}\label{sec:KNN}

For a final set of experiments, we evaluate the efficacy of the TVG model as the first step in a supervised machine learning pipeline, which ends with the feature of the degree-1 zigzag barcode.
For these experiments, we are interested in a binary classification problem with samples from two classes, described below.
\begin{itemize}
    \item \textbf{Class Earth-Moon}: Consists of random samples of $s=15$ STARLINK satellites around the Earth and 5 satellites around the Moon, simulated over one day (86400 seconds) and converted into a $20\times 20$ TVG matrix.
    \item \textbf{Class Earth-Mars}: Consists of random samples of $s=15$ STARLINK satellites around the Earth and 5 satellites around Mars, simulated over one day and converted into a $20\times 20$ TVG matrix.
\end{itemize}

In \Cref{fig:KNN-single-day} we use the same day in March 2023 to generate all the simulations, but we also vary the number $s$ of STARLINK satellites in the system from 5 to 50.
We do this for two reasons: (1) because the number of STARLINK satellites is the primary driver of ``noise'' in our TVGs, as they are being chosen from a database with over 2500 STARLINK satellites, and (2) because we want to understand the computational complexity of these simulations.
In particular, varying $s$ leads to TVG matrices of size $10\times 10$ (for $s=5$), $15\times 15$ (for $s=10$), $20\times 20$ (for $s=15$), and $55\times 55$ (for $s=50$).
Using a current state-of-the-art laptop (a 2023 MacBook Pro, 12‑core CPU and 38‑core GPU M2 Max Chip) generating the \texttt{.orb} files took only a few seconds for each experiment (100 random simulations for each value of $s$) and generating the list of contact times using SOAP only took two minutes.
The main computational bottleneck was in computing zigzag persistence for each simulation and constructing the $100\times 100 $ matrix of $W_2$-distances between these barcodes. For $s=15$  the distance matrix only took ten minutes to calculate, but for $s=50$ nodes the analogous computation took 193 minutes.
Nevertheless, the classification results are all very good with $k=3$ providing $\sim 98\% $ classification accuracy for $s=5$ and $s=10$, and $\sim 93\% $ accuracy for $s=15$ and $s=50$.

Fixing $s=15$, we can also vary the days that the simulation takes place on.
This has a much stronger effect on the classification performance since the relative positions of Earth and Mars, as measured by its synodic period, has a $\sim 2$ year (780 day) period.
For example, in the top row of \Cref{fig:KNN} a random day is chosen from each month, beginning in January 2018 and continuing for 24 months, and a pair of Earth-Moon and Earth-Mars simulations are generated. This yields 48 samples, evenly distributed across the two classes. For a particular $80\%-20\%$ split, the classification performance of KNN is shown in the top-left of \Cref{fig:KNN} with a peak classification rate occuring at $k=5$ neighbors.
The top-right of \Cref{fig:KNN} is calculated by average these performance curves across 200 train-test splits on this same $n=48$ sample experiment.
In the bottom row of \Cref{fig:KNN} a similar experiment is repeated, except for each randomly sampled day 2 simulations are generated from each class, thus leading to a $n=96$ experiment. Notice that the classification power increases as different days are being balanced by random draws from the same day.
Finally, in \Cref{fig:knn-over-the-years} this same basic experiment is repeated, but with different start years: 2021, 2022, 2023, and 2024.
Average KNN classification performance is generally above 
80\% for $k < 5$ for all of these years.
% Average KNN classification performance peaks is best for simulations starting in 2018 and 2020 with over 80\% accuracy (versus a naive 50\% classification rate). Simulations starting in 2016, 2022, and 2024 fare worse with peak classification performance of around 65-72\%.

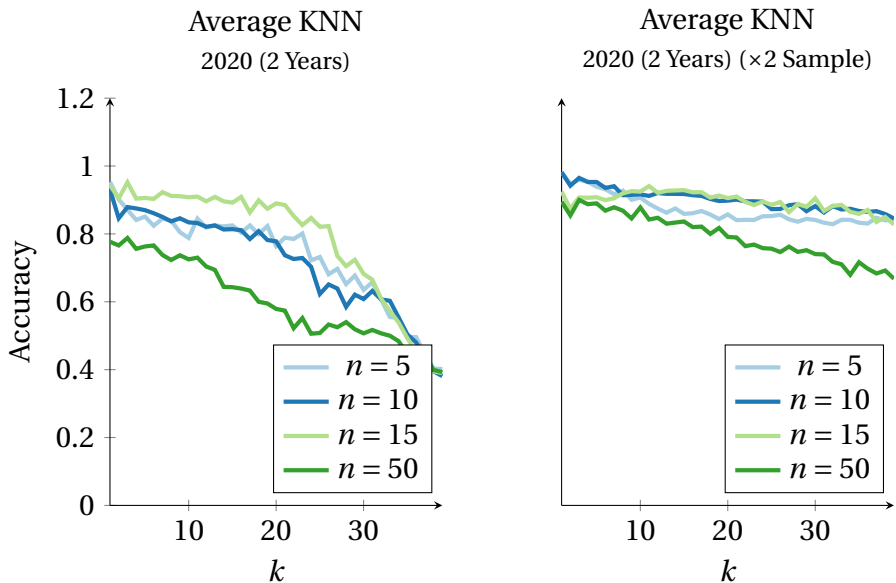
\begin{figure}[h]
\begin{center}
\begin{tikzpicture}[baseline]
\begin{axis}[
    width=6cm,
    height=7cm,
    title={Average KNN \\
        {\footnotesize 2020 (2 Years)}},
    xlabel={$k$},
    ylabel={Accuracy},
    xmin=1, 
    ymin=0,
    ymax=1.2,
    legend pos=south east
    % y tick label style={/pgf/number format/.cd,fixed,precision=5,/tikz/.cd},
    % scaled y ticks = false,
    % axis lines=left,
    % cycle list/Paired
]
% \foreach \yindex in {1,...,3} {
%     \addplot+[ultra thick] table[y index = \yindex]
%         {simulations/knn_single_overlay_today.dat};
% }
\foreach \yindex/\nval in {1/5,2/10,3/15,5/50} {
    \edef\temp{\noexpand\addlegendentry{$n=\nval$}}
    % \addlegendentry{$n=\nval$};
    \addplot+[ultra thick] table[y index = \yindex]
        {simulations/knn_single_overlay_average_2020.dat};
    \temp
}
\end{axis}
\end{tikzpicture}
\qquad \qquad
\begin{tikzpicture}[baseline]
\begin{axis}[
    width=6cm,
    height=7cm,
    title={Average KNN \\
        {\footnotesize 2020 (2 Years) ($\times 2$ Sample)}},
    xlabel={$k$},
    xmin=1, 
    ymin=0,
    ymax=1.2,
    ytick=\empty,
    legend pos=south east
    % y tick label style={/pgf/number format/.cd,fixed,precision=5,/tikz/.cd},
    % scaled y ticks = false,
    % axis lines=left,
    % cycle list/Paired
]
% \foreach \yindex in {1,...,3} {
%     \addplot+[ultra thick] table[y index = \yindex]
%         {simulations/knn_single_overlay_average.dat};
% }
\foreach \yindex/\nval in {1/5,2/10,3/15,5/50} {
    \edef\temp{\noexpand\addlegendentry{$n=\nval$}}
    % \addlegendentry{$n=\nval$};
    \addplot+[ultra thick] table[y index = \yindex]
        {simulations/knn_single_overlay_average_2020_2.dat};
    \temp
}
\end{axis}
\end{tikzpicture}
\end{center}
\caption{KNN performance on $H_1$ zigzag persistence to differentiate satellite systems over a two year period. On the left, 24 days are selected, one from each month in a 24 month period beginning in 2020, and two simulations are generated---one from the Earth-Moon Class and one from the Earth-Mars Class for a total of 48 simulations. On the right, the same experiment is performed, but with 2 samples drawn from each class, leading to 96 simulations. Having more simulations from the same day improves classification accuracy considerably, 
% by approximately 10\%.
}
    \label{fig:KNN}
\end{figure}

% \begin{figure}[h]
%     \centering
%     %\includegraphics[width=0.485\textwidth]{knn-random-46}
%     \includegraphics[width=0.485\textwidth]{knn-random-46-average}    
%     %\includegraphics[width=0.485\textwidth]{knn-random-180}
%     \includegraphics[width=0.485\textwidth]{knn-random-180-average}
%     % \includegraphics[width=0.55\textwidth]{starlink-confidence-mean}
%     \caption{KNN performance on $H_1$ zigzag persistence to differentiate satellite systems over a two year period. On the left, 23 days are selected, one from each month in a 23 month period beginning in 2018, and two simulations are generated---one from the Earth-Moon Class and one from the Earth-Mars Class for a total of 46 simulations. On the right, the same experiment is performed, but with 4 samples drawn from each class, leading to 184 simulations. Having more simulations from the same day improves classification accuracy considerably, by approximately 10\%.}
%     \label{fig:KNN}
% \end{figure}

%\section{Universal semi-ring Models for Space Network
\section{Future Directions}\label{sec:future-dirs}

This paper is one installment in an on-going effort to integrate ideas from mathematics into the design of better engineered systems capable of supporting fast, reliable, and autonomous routing in a solar system-wide internet.
We conclude with a list of ongoing and future work that we hope will draw greater attention to the rich source of research problems that space networking presents.

\begin{enumerate}
    \item \textbf{Faster Algorithms:} Current protocols for routing over a temporal and disruption tolerant network involves the construction of a contact graph and solving Dijkstra's algorithm over this graph; this is called Contact Graph Routing (CGR). Recent efforts \cite{moy2023} show that this can be sped up substantially by using a contact multigraph, as it reduces the number of vertices that Dijkstra's algorithm needs to search over. This mult-graph can be viewed as a display space associated to the cosheaf that defines a TVG. Do the perspectives introduced in this paper provide further computational benefits? For example, is there a routing algorithm that operates directly on the matrix TVG, viewed as a data structure?
    The current goal for space simulation work is the ability to simulate routing over a network with $10^5$ nodes, with our current capabilities hovering around $10^3$.
    \item \textbf{Exploiting Periodicity:} From an algebraic perspective, the difficulty of routing in a system with propagation delay (such as in Earth-to-Mars communication) is the lack of convergence of the Kleene star. However, this comes from using $\End(\powR)$ as a semi-ring. In truly periodic systems, such as those governed by celestial mechanics, it may make more sense to view lifetimes as subsets of the circle $\mathbb{S}^1$. If we consider the communication matrix $A$ with entries in $\End(\mathcal{P}(\mathbb{S}^1))$, there may be conditions when non-trivial propagation delay may still yield a convergent Kleene star, such as in when a shift is a rational multiple of the circumference.
    \item \textbf{Sub-Netting in TVGs:} Our KNN classification experiments indicate that TVGs, even when reduced to their zigzag barcodes, can represent distinct types of space networking scenarios. This leads us to believe that the metrics introduced in this paper can be used to identify motifs or other sub-structures in TVGs that can be used to automatically identify domains for routing. This is important as terrestrial internet relies heavily on fixed routing tables, which is not possible in a time-varying setting. Is it possible to identify algorithms or methods for automatically sub-netting a TVG? If so, this would help make routing protocols adapted to the type of (sub) TVG at hand, which could then be composed to efficiently route across disparate parts of a future space internet.
    \item \textbf{More Realistic Models:} The approach to satellite network modelling taken in this paper is based on considering all possible contacts based on line-of-sight. In practice, many (or most) nodes might only be able to hold one link at a time; choosing one contact means not choosing others. Moreover setting up a link, tearing down a link, and pointing antennas all take time. This means that a realistic network is really based on a subset of all possible contacts. Methods to choose such a subset, particularly in a resilient and robust way, remain unknown. The machinery introduced in this paper could be used to create ``best practices'' and frameworks for network management.
\end{enumerate}

% Many groups of people are working to develop a high-rate delay tolerant networking (HDTN) protocol that is adapted to the challenges and opportunities of a solar system-wide internet.
% In this paper we create a class of ``light weight'' mathematical models for fundamental ideas in HDTN such as time-varying graphs, propagation delay, store-and-forward, all using semi-rings.
% For time-varying graphs, our model has immediate connections to zigzag persistence, which can be viewed as both a form of featurization and a form of dimensionality reduction.
% There are still 

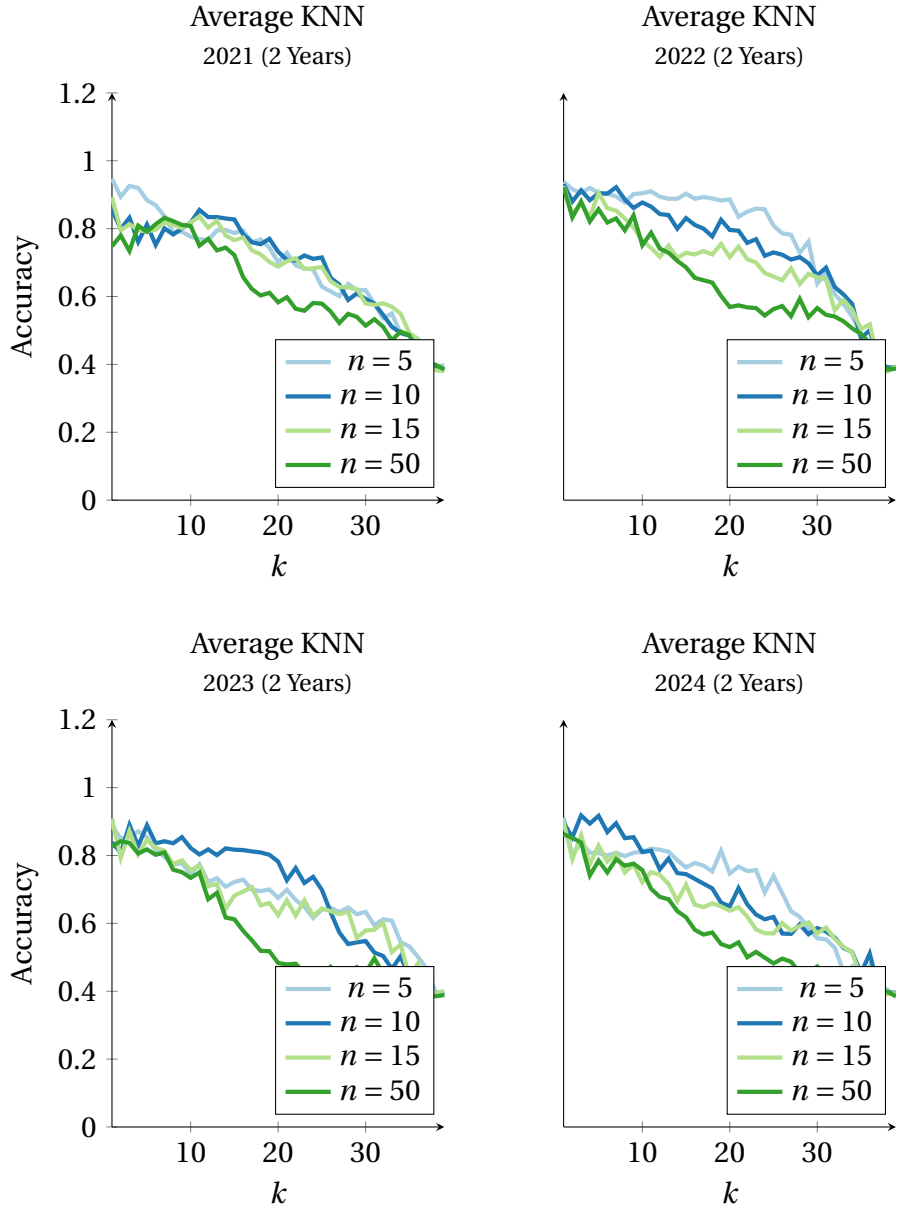
\begin{figure}[h!]
\begin{center}
\begin{tikzpicture}[baseline]
\begin{axis}[
    width=6cm,
    height=7cm,
    title={Average KNN \\
        {\footnotesize 2021 (2 Years)}},
    xlabel={$k$},
    ylabel={Accuracy},
    xmin=1, 
    ymin=0,
    ymax=1.2,
    legend pos=south east
    % y tick label style={/pgf/number format/.cd,fixed,precision=5,/tikz/.cd},
    % scaled y ticks = false,
    % axis lines=left,
    % cycle list/Paired
]
% \foreach \yindex in {1,...,3} {
%     \addplot+[ultra thick] table[y index = \yindex]
%         {simulations/knn_single_overlay_today.dat};
% }
\foreach \yindex/\nval in {1/5,2/10,3/15,5/50} {
    \edef\temp{\noexpand\addlegendentry{$n=\nval$}}
    % \addlegendentry{$n=\nval$};
    \addplot+[ultra thick] table[y index = \yindex]
        {simulations/knn_single_overlay_average_2021.dat};
    \temp
}
\end{axis}
\end{tikzpicture}
\qquad \qquad
\begin{tikzpicture}[baseline]
\begin{axis}[
    width=6cm,
    height=7cm,
    title={Average KNN \\
        {\footnotesize 2022 (2 Years) }},
    xlabel={$k$},
    xmin=1, 
    ymin=0,
    ymax=1.2,
    ytick=\empty,
    legend pos=south east
    % y tick label style={/pgf/number format/.cd,fixed,precision=5,/tikz/.cd},
    % scaled y ticks = false,
    % axis lines=left,
    % cycle list/Paired
]
% \foreach \yindex in {1,...,3} {
%     \addplot+[ultra thick] table[y index = \yindex]
%         {simulations/knn_single_overlay_average.dat};
% }
\foreach \yindex/\nval in {1/5,2/10,3/15,5/50} {
    \edef\temp{\noexpand\addlegendentry{$n=\nval$}}
    % \addlegendentry{$n=\nval$};
    \addplot+[ultra thick] table[y index = \yindex]
        {simulations/knn_single_overlay_average_2022.dat};
    \temp
}
\end{axis}
\end{tikzpicture}
\end{center}

\begin{center}
\begin{tikzpicture}[baseline]
\begin{axis}[
    width=6cm,
    height=7cm,
    title={Average KNN \\
        {\footnotesize 2023 (2 Years)}},
    xlabel={$k$},
    ylabel={Accuracy},
    xmin=1, 
    ymin=0,
    ymax=1.2,
    legend pos=south east
    % y tick label style={/pgf/number format/.cd,fixed,precision=5,/tikz/.cd},
    % scaled y ticks = false,
    % axis lines=left,
    % cycle list/Paired
]
% \foreach \yindex in {1,...,3} {
%     \addplot+[ultra thick] table[y index = \yindex]
%         {simulations/knn_single_overlay_today.dat};
% }
\foreach \yindex/\nval in {1/5,2/10,3/15,5/50} {
    \edef\temp{\noexpand\addlegendentry{$n=\nval$}}
    % \addlegendentry{$n=\nval$};
    \addplot+[ultra thick] table[y index = \yindex]
        {simulations/knn_single_overlay_average_2023.dat};
    \temp
}
\end{axis}
\end{tikzpicture}
\qquad \qquad
\begin{tikzpicture}[baseline]
\begin{axis}[
    width=6cm,
    height=7cm,
    title={Average KNN \\
        {\footnotesize 2024 (2 Years) }},
    xlabel={$k$},
    xmin=1, 
    ymin=0,
    ymax=1.2,
    ytick=\empty,
    legend pos=south east
    % y tick label style={/pgf/number format/.cd,fixed,precision=5,/tikz/.cd},
    % scaled y ticks = false,
    % axis lines=left,
    % cycle list/Paired
]
% \foreach \yindex in {1,...,3} {
%     \addplot+[ultra thick] table[y index = \yindex]
%         {simulations/knn_single_overlay_average.dat};
% }
\foreach \yindex/\nval in {1/5,2/10,3/15,5/50} {
    \edef\temp{\noexpand\addlegendentry{$n=\nval$}}
    % \addlegendentry{$n=\nval$};
    \addplot+[ultra thick] table[y index = \yindex]
        {simulations/knn_single_overlay_average_2024.dat};
    \temp
}
\end{axis}
\end{tikzpicture}
\end{center}
% \caption{}
% \caption{KNN performance on $H_1$ zigzag persistence to differentiate satellite systems over a two year period. On the left, 24 days are selected, one from each month in a 24 month period beginning in 2020, and two simulations are generated---one from the Earth-Moon Class and one from the Earth-Mars Class for a total of 48 simulations. On the right, the same experiment is performed, but with 2 samples drawn from each class, leading to 96 simulations. Having more simulations from the same day improves classification accuracy considerably, 
% % by approximately 10\%.
% }
%     \label{fig:KNN}
\caption{KNN average performance using 96 total simulations as described in \Cref{fig:KNN} where the two-year period starts in the indicated year.}
    \label{fig:knn-over-the-years}
\end{figure}

% \begin{figure}
%     \centering
%     %\includegraphics[width=0.35\textwidth]{knn-2016}
%     \includegraphics[width=0.485\textwidth]{knn-2016-average}    
%     %\includegraphics[width=0.35\textwidth]{knn-2018}
%     \includegraphics[width=0.485\textwidth]{knn-2018-average}
%     %\includegraphics[width=0.35\textwidth]{knn-2020}
%     \includegraphics[width=0.485\textwidth]{knn-2020-average}
%     %\includegraphics[width=0.35\textwidth]{knn-2022}
%     \includegraphics[width=0.485\textwidth]{knn-2022-average}
%     %\includegraphics[width=0.35\textwidth]{knn-2024}
%     \includegraphics[width=0.485\textwidth]{knn-2024-average}
%     % \includegraphics[width=0.55\textwidth]{starlink-confidence-mean}
%     \caption{KNN average performance using 184 total simulations as described in \Cref{fig:KNN} where the two-year period starts in the indicated year.}
%     \label{fig:knn-over-the-years}
% \end{figure}

\newpage
%\section*{References}

\Urlmuskip=0mu plus 1mu\relax
\printbibliography

\appendix
\section{Definitions and Technical Results on Semi-Rings}\label{sec:semi-ring-appendix}

\begin{definition}\label{def:monoid}
    A \define{monoid} is a set $(S,\circ,\e)$ equiped with a binary operation $\circ:S\times S \to S$, that is associative, i.e., $a\circ(b\circ c) = (a\circ b)\circ c$, and where $\e$ acts as a neutral element on both the left and the right, i.e., $a\circ \e = a = \e \circ a$.
    Such a monoid is \define{commutative} if $a\circ b =b\circ a$ for all pairs of elements $a,b\in S$.
\end{definition}

\begin{definition}\label{def:semi-ring-homomorphism}
An \define{additive homomorphism} from a semi-ring $(S,\oplus,\otimes,\n,\e)$ to $(T,+,\times,0,1)$ is a map $h:S\to T$ such that
\begin{itemize}
\item $h(\n)=0$, and
\item $h(a\oplus b)=h(a) + h(b)$.
\end{itemize}
A \define{homomorphism} of semi-rings is an additive homomorphism that also satisfies the properties
\begin{itemize}
\item $h(\e)=1$, and
\item $h(a\otimes b)=h(a)\times h(b)$.
\end{itemize}
An additive endomorphism is an additive homomorphism of $S$ to itself.
Similarly, an endomorphism is a homomorphism from $S$ to itself.
\end{definition}

On \cite[pg. 14]{Baras2010} we have the following elementary observation:

\begin{lemma}\label{lem:end-semi-ring}
If $S$ is a semi-ring, then $H:=\End(S)$, the set of additive endomorphisms of $S$, is a semi-ring as well. The two operations are
\begin{itemize}
\item $(h_1+h_2)(a):=h_1(a)+h_2(a)$
\item $(h_1\odot h_2)(a):=h_2(h_1(a))$
\end{itemize}
Notice that, in particular, we did not need to assume that endomorphisms preserve products to make $\End(S)$ into a semi-ring. Also notice that even if $S$ is a commutative semi-ring, i.e., $\otimes$ is commutative, then $\End(S)$ need not be commutative.
\end{lemma}

Below is the proof of Theorem \ref{thm:boolean-curry}, which we restate for convenience.

\begin{theorem}\label{thm:boolean-curry-appendix}
The semi-ring of matrix TVGs $\,\Mat_n(\powR)$ is isomorphic to the semi-ring of functions from $\R$ to $\Mat_n(\Bool)$.
This isomorphism is witnessed by the homomorphism
\[
    \Psi: \Mat_n(\powR) \to \Fun(\R,\Mat_n(\Bool)), \quad \text{where} \quad \Psi(M)(t)=\snap_t(M)
\]
is the snapshot of $M$ at time $t$.
\end{theorem}
\begin{proof}
$\Psi$ preseves neutral elements. The identity TVG matrix $\R I$, with diagonal entries $\R$, is sent to the function $1$ that assigns to each $t$ the identity Boolean matrix $\top I$. The TVG matrix of empty sets $\varnothing I$ is sent to the function $0$ with constant value $\bot I$.

To establish clear notation, we will write TVG matrices with capital letters like $M$ and $N$ and use $M+N$ and $MN$ for matrix addition and multiplication.
To emphasize that the entries of these matrices are subsets, we use union and intersection to emphasize the operations occuring in $\powR$, i.e.,
\[
(M+N)_{ij}=M_{ij}\cup N_{ij} \quad \text{and} \quad (MN)_{ij}=\bigcup_k M_{ik}\cap N_{kj}.
\]
In the semi-ring $\Fun(\R,\Mat_n(\Bool))$ we will use lower case letters like $m$ and $n$ to stand for functions. 
We write $m(t)$ for the matrix of Booleans at time $t$, whose truth value in entry $(i,j)$ is denoted $m(t)(i,j)$. 
Pointwise addition of functions is written $(m\oplus n)(t)=m(t)+n(t)$, where we re-use the $``+''$ operator for addition of Boolean matrices. The truth value of $(m\oplus n)(t)(i,j)$ is determined by $m(t)(i,j)\vee n(t)(i,j)$.
Multiplication in $\Fun(\R,\Mat_n(\Bool))$ is pointwise, i.e., $(m\otimes n)(t)=m(t)n(t)$, where again we use matrix mutliplication notation. The truth value of $(m\otimes n)(t)(i,j)$ is determined by $\vee_k m(t)(i,k)\wedge n(t)(k,j)$.
To alleviate notation slightly, we will write $\Psi(M)=m$, which is the unique function that satisfies the property 
\[
    t\in M_{ij} \iff m(t)(i,j)=\top.
\]
% This determines a bijection between $\Mat_{n\times n}(\mathcal{P}(\R))$ and $\Fun(\R,\Mat_{n\times n}(\Bool))$ at the level of sets, so we only need to ensure that the semi-ring operations are preserved under this bijection.

To see that the assignment $\Psi:M \mapsto m$ preserves addition, note that $t \in (M+N)_{ij}$ iff $t\in M_{ij}$ or $t\in N_{ij}$, which is true iff
\[
m(t)(i,j)\vee n(t)(i,j)=(m(t) + n(t))(i,j)=(m\oplus n)(t)(i,j)=\top.
\]
Similarly, to see that the assignment $\Psi:M \mapsto m$ preserves multiplication, note that
\begin{eqnarray*}
t\in (MN)_{ij} & \iff & t\in \bigcup_k M_{ik}\cap N_{kj} \\
& \iff & \top = \bigvee_k m(t)(i,k)\wedge n(t)(k,j) \\
& \iff & \top = (m(t)n(t))(i,j) \\
& \iff & \top = (m\otimes n)(t)(i,j).
\end{eqnarray*}
\end{proof}

We now state and prove Proposition \ref{prop:Kleene-star-diam}.

\begin{proposition}[Convergence at the Temporal Diameter]\label{prop:Kleene-star-diam-appendix}
Let $A$ be the adjacency matrix for a simple TVG $\cG = (G, \ell_M)$.
The $(i,j)^{th}$ entry of the $k$-cumulant $C_k(A)$ stabilizes after $d_{ij}$, where
\[
d_{ij}= \max_{t \mid \exists \gamma:i\rightsquigarrow j} \min_{\gamma} |\gamma|
\]
is the length of the longest shortest path from $i$ to $j$, disregarding times where no such path exists. 
We set $d_{ii}=0$, by convention.
We define the \define{temporal diameter} of a TVG $\cG$ to be
\[
\diam (\cG)=\max_{ij} d_{ij}.
\]
Consequently, the Kleene star $A^*$ convergences for $r\geq \diam (\cG)$.
\end{proposition}
\begin{proof}
By Remark \ref{rmk:matrix-powers}, the question of whether or not an entry of $C_k(A)$ stabilizes at some value of $k$ is equivalent to asking if that entry of $M^k=(I+A)^k$ stabilizes at $k$.
Moreover, by Theorem \ref{thm:boolean-curry}, we can faithfully recover the values of $M^k$ by considering the associated Boolean function $m^k:V\times V\times \R \to \Bool$, when is $\Psi(M^k)=\Psi(M)^k$.
This is the logical \texttt{OR} (union) of $a^p$, which is $\Psi(A^p)=\Psi(A)^p$, across all $p=1,\dots,k$ along with the identity relation, thus it suffices to characterize what nodes are related at time $t$ using $a^p(t):V\times V \to \Bool$.
For $p=1$, this is simple: $a(i,j,t)=1$ if there is an edge from $i$ to $j$ at time $t$.
By direct inspection of matrix multiplication, one can see that $a^2(i,j,t)=1$ if and only if there is a length-2 walk at time $t$.
More generally, $a^p(i,j,t)=1$ if and only if there is a length-$p$ walk from $i$ to $j$ at time $t$. If $i$ and $j$ are connected by some walk, then $d_{ij}$ is defined to be the maximum over all $t\in\R$. 
Since the maximum diameter of any connected component of a graph on $|V|$-vertices is $|V|-1$, and $a^p(i,j,\bullet):\R \to \Bool$ is piece-wise constant on finitely many pieces (by virtue of $\subR$ rather than $\powR$), this maximum necessarily exists.
\end{proof}

We now prove Proposition \ref{prop:UCS-sub-quot}.

\begin{proposition}\label{prop:UCS-sub-quot-appendix}
Let $\mathcal C$ denote the Universal Contact Semi-Ring (Definition \ref{def:UCS}). $\mathcal C$ contains 
    \begin{enumerate}
        \renewcommand{\theenumi}{(\alph{enumi})}
        \item an injective image of the boolean semi-ring $S$;
        \item a sub-semi-ring which is isomorphic to the TVG semi-ring;
        \item a sub-semi-ring which surjects onto the tropical semi-ring $\mathbb T$;
        \item a sub-semi-ring which surjects onto the propagation delay semi-ring ; and
        \item a sub-semi-ring which surjects onto the function endomorphism semi-ring $\mathbb W$.
    \end{enumerate}
\end{proposition}
    \begin{proof}

        \begin{enumerate}
        \renewcommand{\theenumi}{(\alph{enumi})}
            \item Consider the map $S \to \mathcal C$, that sends $0 \mapsto \left(t \mapsto \emptyset\right)$ and $1 \mapsto \left(t \mapsto \{0\}\right)$. It is clear that the map is a injective semi-ring homomorphism, hence the claim follows.
            \item On the UCS, consider the sub-semi-ring that consists of maps 
                \[ 
                    t \mapsto 
                    \begin{cases} 
                        \{0\} & \text{ for }  t \in I, \\ 
                        \emptyset & \text{ otherwise, } 
                    \end{cases}\]
                where $I$ is a finite union of closed, disjoint intervals of $\R$. 
                This sub-semi-ring is isomorphic the TVG semi-ring via the map which sends the above maps to the lifetime $I$.
            \item Let 
                $\left\{ t \mapsto I, I \text{ is finite} \right\}$ be a collection of constant maps in $\mathcal C$ that map to finite subset of $\R$. 
                This collection forms a sub-semi-ring of $\mathcal C$ which surjects onto the tropical semi-ring, via the semi-ring homomorphism:
                \[  (t \mapsto I)\mapsto \begin{cases}\min(I), &  I \ne \emptyset \\ \infty, & \text{ otherwise.}\end{cases} \]
            \item Consider a collection $\left\{\varphi:\R \to 2^\R \right\}\cup\left\{0_{\mathcal C}: t \mapsto \emptyset\right\} \subseteq \mathcal C$, where $\varphi(t)$ is a finite subset of $\R^+$ for all $t\in \R$ such that $\sup\left\{\bigcup_{t \in \R} \varphi(t)\right\} < \infty$.
                This collection is a sub-semi-ring of $\mathcal C$ and surjects onto the Propagation Delay semi-ring via the \textit{additive} homomorphism that sends:
                          \[ \left(\varphi: \R \to 2^\R \right) \mapsto (K,m),\]
                where $K = \left\{t : \varphi(t) \ne \emptyset\right\}$ and $m = \sup\left\{\bigcup_{t \in \R} \varphi(t)\right\}$.
            \item Similar to the tropical semi-ring case, consider the collection of maps $\left\{\varphi \right\}$ in $\mathcal C$ that satisfy
                    \begin{itemize}
                        \item $\varphi(t)$ is a finite subset of $[0, \infty)$ for each $t$,
                            \item If $t_1 \le t_2$ and $\varphi(t_1), \varphi(t_2)$ are not empty, then $\min(\varphi(t_2)) + t_2 \ge \min(\varphi(t_1)) + t_1$, and
                            \item If $t_1 \le t_2$ and $\varphi(t_1) = \emptyset$ then $\varphi(t_2) = \emptyset$.
                     \end{itemize}
    This collection is a sub-semi-ring of $\mathcal C$, and surjects onto $\mathbb W$ via the semi-ring homomorphism:
    \[ \varphi \mapsto \left(t \mapsto \begin{cases}\min(\varphi(t)) + t & \varphi(t) \ne \emptyset\\ \infty & \varphi(t) = \emptyset\end{cases}\right). \]
        \end{enumerate}
    \end{proof}

\section{Basic Category Theory}\label{sec:basic-category-theory}

A \define{category} $\mathcal{C}$ consists of a collection of \emph{objects}, $\text{Ob}, \C$ and a collection of \emph{morphisms}, $\hom (x, y)$, for every $x, y$ in $\text{Ob} \C$, such that
\begin{itemize}
        \item if $f$ in $\hom (x, y)$ and g $\hom (y, z)$, then $g \circ f$ in $\text{hom}(x, z)$,
        \item for each $x$ in $\text{Ob} \C$, there exists a morphism $\id_x$ in $\text{hom}(x, x)$, 
        \item for $f$ in $\hom (x, y)$, $f\circ \id_x = \id_x \circ f$, and
        \item composition is associative: $(f \circ g) \circ h = f \circ (g \circ h)$.
\end{itemize}
A \define{functor} $\mathcal F$ between categories $\C$ and $\D$ consists of
\begin{itemize}
    \item an object $F(x) \in \D$ for each object $x \in \C$, 
    \item a morphism $F(f) \in \hom(F(x),F(y))$ for each morphism $f\in \hom(x, y)$, such that
    \begin{itemize}
        \item $F$ respects the composition rule in $\C$ and $\D$:
        \[F(f\circ g) = F(f)\circ F(g), \]
        \item for each object $x$ in $\C$, $F(\id_x) = \id_{F(x)}$.
    \end{itemize}
\end{itemize}

\section{Barcodes and Distances on These}\label{sec:barcode-review}
A \define{barcode} $\mathcal B$ is a multiset of intervals $\{I\}$ in $\R$. 
A matching $\sigma$  between barcodes $\mathcal B$ and $\mathcal C$ is defined as a partial bijection between 
 $\mathcal{B}' \subset \mathcal{B}$,and $\mathcal{C}' \subset \mathcal{C}$:
 \[\sigma: \mathcal{B'} \to \mathcal{C'}.\]
Given a $p \in [1,\infty]$ we assign a\define{matching $p$-cost} of $\sigma$ as follows
\[
\cost{\sigma, p} = \left\{
\begin{aligned}
        \left( \sum\limits_{\substack{ I \in \mathcal B, \\ \sigma(I) = J }}{\norm{I - J}_p^p + }  \sum\limits_{I \in \Delta}{\norm{I - \Mid(I)}}_p^p \right)^{1/p}, & \text{ when } 1 \leq p < \infty \\
        \max\left\{ \max\limits_{\substack{ I \in \mathcal B, \\ \sigma(I) = J }}{\norm{I - J}_\infty, \max\limits_{I \in \Delta }{\norm{I - \Mid(I)}}_\infty } \right\}, & \text{ when } p = \infty,
\end{aligned}
\right \},
\]
where $\norm{I-J}_p$ is the $\ell^p$-norm between intervals $I = [a, b)$ and $J=[c,d)$ viewed as vectors $(a, b)$ and $(c,d)$ in $\R^2$, $\Mid(I) := \left[\tfrac{a+b}{2}, \tfrac{a+b}{2}\right)$ is the empty interval at the midpoint of $I$, and $\Delta$ denotes all the intervals unmatched by $\sigma$ in $\mathcal B \sqcup \mathcal C$.
The \define{Wasserstein $p$-distance} between barcodes $\mathcal B$ and $\mathcal C$ is then defined as the infimum of $p$-costs over all possible matchings, i.e.,
\[
d^p_{W}(\mathcal B, \mathcal C) = \inf\limits_{\sigma}{\cost{\sigma, p}}.
\]
The distance $d^\infty_{W}$ is called the \define{bottleneck distance}, denoted as $d_B$.

\section{Proof of Isometry Theorems}\label{sec:isometry-theorem-appendix}

%\robert{I think the interleaving }

%\begin{theorem}[Isometry theorem] For any pair of time-varying graphs $\cG, \cH$ of $n$ nodes, the interleaving distance and the symmetrized Hausdorff distance coincide, i.e 
%    \[
%    d_{\Sigma H} (\cG, \cH) = d_{I} (\cG, \cH)
%    .\] 
%\end{theorem}

We now prove Theorem \ref{thm:isometries}.

\begin{theorem}[Isometry Theorems]\label{thm:isometries-appendix}
Suppose $\cG=(G,\ell_M)$ and $\cG'=(G,\ell_N)$ are two TVGs with the same underlying graph.
We can view $\mathcal{E}_N$ and $\boldsymbol{\iota}\circ \mathcal{E}_N$ as functors to $\sub(G)$ and $\mon(G)$, respectively.
Any ordering of the nodes in $G$ then determines two isometries
    \[
    d_{H} (M, N) = d_{I} (\mathcal{E}_M, \mathcal{E}_N) \quad \text{and} \quad d_{\Sigma H} (M, N) = d_{I} (\boldsymbol{\iota}\circ\mathcal{E}_M, \boldsymbol{\iota}\circ\mathcal{E}_N),
    \] 
where $M$ and $N$ are the lifetime matrices for $\cG=(G,\ell_M)$ and $\cH=(H,\ell_N)$, respectively.

\begin{proof}
Suppose $G$ has $n$ vertices and fix a labeling for them $\{1, \ldots, n\}$.

We first show 
$d_{\Sigma H}(M, N) \leq d_I(\iota \cdot \mathcal E_M, \iota \cdot \mathcal E_N)$.
Suppose $(\varphi, \psi)$ is a $\varepsilon$-interleaving, by taking colimits we can
see that $\varphi$ corresponds to some permutation $\sigma \in \Pi_n$.
We want to show that for all $i, j$, we have 
$M_{ij} \subseteq N_{\sigma(i) \sigma(j)}^\varepsilon$ and 
$N_{\sigma(i) \sigma(j)} \subseteq M_{ij}^\varepsilon$.
If $[i j] \in \mathcal E_M(I)$ then $\varphi_I$ guaranteees 
$[{\sigma (i)} {\sigma (j)}] \in \mathcal E_N (I^\varepsilon)$.
In particular, there exists $t \in I$ such that $t \in M_{ij}$. 
Since $\varphi_I$ guarantees that $[{\sigma (i)} {\sigma (j)}] $ exists in $\mathcal E_N (I^\varepsilon)$,
there exists $t' \in N_{\sigma(i) \sigma(j)}$ with 
$t' \in (t - \varepsilon, t + \varepsilon)$. 
This shows $t \in N_{\sigma(i) \sigma(j)}^\varepsilon$.
A symmetric argument shows the other inclusion.

We now show 
$d_I(\iota \cdot \mathcal E_M, \iota \cdot \mathcal E_N) \leq d_{\Sigma H}(M, N)$.
Let $\sigma$ be a permutation in $\Pi_n$ and consider 
$d_H(M, \sigma(N) ) = \varepsilon$, that is, for all $i, j$, we have 
$M_{ij} \subseteq N_{\sigma(i) \sigma(j)}^\varepsilon$ and 
$N_{\sigma(i) \sigma(j)} \subseteq M_{ij}^\varepsilon$. 
For $I \in \Int$, define $\varphi_I : \mathcal E_M (I) \to \mathcal E_N (I^\varepsilon)$
by $i \mapsto {\sigma(i)}$ and $[i j] \mapsto [{\sigma(i)} {\sigma(j)}]$.
If $e = [i j] \in \mathcal E_M (I)$, then there exists $t \in I$ such that 
$t \in M_{ij}$. But since $M_{ij} \subseteq N_{\sigma(i) \sigma(j)}^\varepsilon$, 
then $t \in N_{\sigma(i) \sigma(j)}^\varepsilon$ which means there exists 
$t' \in N_{\sigma(i) \sigma(j)}$ such that $\vert t - t' \vert < \varepsilon$ which
means that $[{\sigma (i)} {\sigma (j)}] \in \mathcal E_N (I^\varepsilon)$, hence $\varphi_I$ is 
well-defined. A symmetric argument shows $\psi_J$ is also well-defined.

The first isometry follows from the second.
\end{proof}
\end{theorem}

%\section*{Acknowledgments}
%We would like to acknowledge the assistance of volunteers in putting
%together this example manuscript and supplement.

% \newpage
%\section{Acknowledgement}

%\section*{References}

\end{document}